\documentclass{amsart}
\usepackage{amsmath,amssymb,dsfont,amsthm}
\usepackage{graphics}
\usepackage{xypic}
\usepackage{latexsym}
\usepackage{amsmath}
\usepackage{amsfonts}
\usepackage{amssymb}
\usepackage{tensor}

\date{\today}

\title[The analytic torsion of a cone over a sphere]{The analytic torsion of a cone over a sphere}

\author{L. Hartmann  and M. Spreafico}
\address{\tt ICMC, Universidade S\~{a}o Paulo, S\~{a}o Carlos, Brazil.}
\email{mauros@icmc.usp.br}

\newtheorem{theo}{Theorem}
\newtheorem{lem}{Lemma}
\newtheorem{corol}{Corollary}
\newtheorem{defi}{Definition}
\newtheorem{conj}{Conjecture}

\renewcommand{\Re}{{\rm Re}}

\newcommand{\Sp}{{\rm Sp}}
\newcommand{\beq}{\begin{equation}}
\newcommand{\eeq}{\end{equation}}

\newcommand{\Z}{{\mathds{Z}}}
\newcommand{\R}{{\mathds{R}}}
\newcommand{\C}{{\mathds{C}}}

\newcommand{\T}{{\mathcal{T}}}

\newcommand\e{{\rm e}}
\newcommand{\A}{{\mathcal{A}}}

\renewcommand{\P}{{\mathcal{P}}}
\newcommand{\B}{{\mathcal{B}}}

\renewcommand{\b}{{\partial}}
\newcommand{\Vol}{{\rm Vol}}
\newcommand{\ec}{{\mathsf e}}
\renewcommand{\ge}{{\mathsf g}}

\date{}

\DeclareMathOperator*{\Rz}{Res_0}

\DeclareMathOperator*{\Ru}{Res_1}

\begin{document}


\maketitle

\begin{abstract} We compute the analytic torsion of a cone over a sphere of dimension 1, 2, and 3, and we conjecture a general formula for the cone over an odd dimensional sphere.
\end{abstract}

\section{Introduction}
\label{s0}

An important open problem in geometric and global analysis is to extend the Cheeger M\"uller theorem to spaces with singularities of conical type. The aim of this work is to give some contribution to the quantitative aspect of the problem. For we give explicit formulas for the analytic torsion of the class of low dimensional spaces consisting of cones over  spheres. The results cover also the smooth case of the discs, and therefore provides also a contribution to the discussion on the extension of the Cheeger M\"uller theorem to smooth manifolds with boundary, namely to the problem of establish the correct boundary term.

Let $(W,g)$ be a closed connected  Riemannian manifold of dimension $n$ with metric $g$. Let $C W$ denote the completed finite metric cone over $W$, namely the space $[0,l]\times W$, with the metric $dr\otimes dr+r^2 g$, on $(0,l]\times W$, as defined in \cite{Che0} (2.1). A interesting open problem concerning the metric cone is to compute its analytic torsion. The analytic torsion of a smooth connected Riemannian manifold $(M,g)$ of dimension $m$ is defined by \cite{RS}, Section 6,
\beq\label{analytic}
\log T(M)=\frac{1}{2}\sum_{q=1}^m(-1)^q q\zeta'(0,\Delta^{(q)}),
\eeq
where $\Delta^{(q)}$ is the Laplace operator on $q$-forms on $M$, and the zeta function is defined by \cite{RS} (1.5)
\[
\zeta(s,\Delta^{(q)})=\sum_{\lambda\in \Sp_+\Delta^{(q)}}\lambda^{-s},
\]
for $\Re(s)>\frac{m}{2}$, and by analytic continuation elsewhere. This definition extends to the case of a cone $CW$ using the Hodge theory and the functional calculus for the Laplace operator on forms developed in \cite{Che0}. 
More precisely, one would like to obtain formulas for $T(CW)$ as a function of some geometric invariant of $W$. Starting from the result of Cheeger \cite{Che0} \cite{Che2}, and applying absolute or relative boundary conditions \cite{RS}, Section 3, one obtain quite easily the eigenvalues of the Laplace operator on forms, necessary to compute the torsion. These eigenvalues turn out to be sequences of real numbers $\Sp_+\Delta^{(q)}=\{\lambda^{(q)}_{\mu,k}\}$ that correspond to the zeros of some linear combinations of Bessel functions of the first kind and their derivative. The index $k$ enumerate the zero, and the index $\mu$ is given by some explicit function of the eigenvalues of the Laplace operator on forms on the section of the cone, namely on $W$. The zeta function of this type of double sequences can be tackled using some recent results of Spreafico \cite{Spr3} \cite{Spr5} \cite{Spr6} \cite{Spr9}. The general strategy is to prove that the sequence $\Sp_+\Delta^{(q)}$ is spectrally decomposable over some sequence $\Sp_+ \Delta_{W}^{(p)}$ of eigenvalues of the Laplacian on forms on the section. Then, one can apply the result of Spreafico to obtain the value $\zeta'(0,\Delta^{(q)})$. The final formula can be very complicate in general, and not particularly illuminating. The possibility of reducing and simplifying this formula is based on two facts: one fact is the explicit form of the coefficients of the uniform asymptotic expansion of the Bessel function $I_\nu(\nu z)$ (and of its derivative) with respect to the order $\nu$; the second fact, is the explicit knowledge  of the eigenvalues of the Laplacian on forms on the section. While the first fact is proved to be true in general, the second one is not clear. For this reason it is interesting to study particular cases where the second fact is also true. 

In this note, we study the analytic torsion of the cone over an $n$-dimensional sphere. More precisely, we prove in Section \ref{s4} the following theorem, and we state a conjecture for the general case at the end of Section \ref{s5}.

\begin{theo} \label{t1} The analytic torsion of the cone $C_\alpha S^{n}_{l\sin\alpha}$ of angle $\alpha$, and length $l>0$, over the sphere $S^{n}$, with the standard metric induced by the immersion in $\R^{n+2}$, and absolute boundary conditions is, for $n=1,2$, and $3$:
\begin{align*}
\log T(C_\alpha S^{1}_{l\sin\alpha})=&\frac{1}{2}\log {\rm Vol}(C_\alpha S^{1}_{l\sin\alpha})+\frac{1}{2}\sin\alpha=\frac{1}{2}\log\pi l^2\sin\alpha+\frac{1}{2}\sin\alpha,\\
\log T(C_\alpha S^{2}_{l\sin\alpha})=&\frac{1}{2}\log {\rm Vol}(C_\alpha S^{2}_{l\sin\alpha})-\frac{1}{2}f({\rm csc}\alpha)+\frac{1}{4}\sin^2\alpha\\
=&\frac{1}{2}\log \frac{4\pi l^3\sin^2\alpha}{3}-\frac{1}{2}f({\rm csc}\alpha)+\frac{1}{4}\sin^2\alpha,\\
\log T(C_\alpha S^{3}_{l\sin\alpha})=&\frac{1}{2}\log {\rm Vol}(C_\alpha S^{3}_{l\sin\alpha})
+\frac{3}{4}\sin\alpha-\frac{1}{12}\sin^3\alpha\\
=&\frac{1}{2}\log \frac{\pi^2l^4\sin^3\alpha}{2}+\frac{3}{4}\sin\alpha-\frac{1}{12}\sin^3\alpha,
\end{align*}
where the function $f(\nu)$ is given at the end of Section \ref{s4}.

\end{theo}

\section{Geometric setup}
\label{s1}

We describe in this section the geometric setup in details. Let $S^n_b$ be the standard sphere of radius $b>0$ in $\R^{n+1}$, $S^{n}_b=\{x\in\R^{n+1}~|~|x|=b\}$ (we simply write $S^n$ for $S^n_1$). Imbed $S^n_{l\sin\alpha}$ in $\R^{n+2}$, with center in the point $\{0,...,0,l\sin\alpha\}$, with $l>0$. Let $C_\alpha
S^n_{l\sin\alpha}$ be the cone of angle $\alpha$ over $S^n_{l\sin\alpha}$ in $\R^{n+2}$. Note that the disc corresponds to $D^{n+1}_l=C_\frac{\pi}{2} S^{n}_l$. We parameterize $C_{\alpha}S^n_{l\sin\alpha}$ by
\begin{equation*}\label{}C_{\alpha}S_{l\sin\alpha}^{n}=\left\{
\begin{array}{rcl}
x_1&=&r \sin{\alpha} \sin{\theta_n}\sin{\theta_{n-1}}\cdots\sin{\theta_3}\sin{\theta_2}\cos{\theta_1} \\[8pt]
x_2&=&r \sin{\alpha} \sin{\theta_n}\sin{\theta_{n-1}}\cdots\sin{\theta_3}\sin{\theta_2}\sin{\theta_1} \\[8pt]
x_3&=&r \sin{\alpha} \sin{\theta_n}\sin{\theta_{n-1}}\cdots\sin{\theta_3}\cos{\theta_2} \\[8pt]
&\vdots& \\
x_{n+1}&=&r \sin{\alpha} \cos{\theta_n} \\[8pt]
x_{n+2}&=&r \cos{\alpha}
\end{array}
\right.\end{equation*}
with $r \in [0,l]$, $\theta_1 \in [0,2\pi]$, $\theta_2,\ldots,\theta_n \in [0,\pi]$, $\alpha$ is a fixed positive real number, and $0<a=\frac{1}{\nu}= \sin{\alpha}\leq 1$. This is a compact connected space. The metric induced by the immersion in $\R^{n+2}$ is 
\begin{align*}
g &=dr \otimes dr + r^2 a^2 g_{S^{n}_{1}},\\
\end{align*}
and is smooth for $r>0$. Comparing with \cite{Che0}, Section 1, we see that the space $C_\alpha S^n_{l\sin\alpha}$ is a completed metric cone, and $X_\alpha=C_\alpha S^n_{l\sin\alpha}-\{0\}$, is a metric cone over $S^n_{l\sin \alpha}$. Note that $X_\alpha$ is not smooth, since the radius of the sphere is not unitary. Note also that the space $C_\alpha S^n_{l\sin\alpha}$ is simply connected (in fact it has the homotopy type of a point).

In order to define the opportune self adjoint extension of the Laplace operator on forms, we split the space of forms near the boundary  as direct sum $\Lambda C_\alpha S^n_{l\sin\alpha}=\Lambda  S^n_{l\sin\alpha}\oplus N^* C_\alpha S^n_{l\sin\alpha}$, where $N^*$ is the dual to the normal bundle to the boundary. Locally, this reads as follows. Let $\b_r$ denotes the outward pointing unit normal vector to the boundary, and $dr$ the correspondent one form. Near the boundary we have the collar decomposition
$C_\alpha S^n_{l\sin\alpha}=[0,-\epsilon)\times  S^n_{l\sin\alpha}$, and if $y$ is a system of local coordinates on the boundary, then $x=(r,y)$ is a local system of coordinates in $C_\alpha S^n_{l\sin\alpha}$. The smooth forms on $C_\alpha S^n_{l\sin\alpha}$ near the boundary decompose as
\[
\omega=\omega_{\rm tan}+\omega_{\rm norm},
\]
where $\omega_{\rm norm}$ is the orthogonal projection on the subspace generated by $dr$, and $\omega_{\rm tan}$ is in $\Lambda  S^n_{l\sin\alpha}$. We  write
\[
\omega=\omega_1+ \omega_{2}\wedge dr,
\]
where $\omega_j\in C^\infty(C_\alpha S^n_{l\sin\alpha})\otimes \Lambda  S^n_{l\sin\alpha}$, and
\[
*\omega_2=dr \wedge *\omega.
\]

Define absolute boundary conditions by
\[
B_{\rm abs}(\omega)=\omega_{\rm norm}|_{S^n_{l\sin\alpha}}=\omega_2|_{S^n_{l\sin\alpha}}=0,
\]
and relative boundary conditions by
\[
B_{\rm rel}(\omega)=\omega_{\rm tan}|_{S^n_{l\sin\alpha}}=\omega_1|_{S^n_{l\sin\alpha}}=0.
\]

Let $\B(\omega)=B(\omega)\oplus B((d+d^\dagger)(\omega))$. Then the operator $\Delta=(d+d^\dagger)^2$ with boundary conditions $\B(\omega)=0$  is self adjoint. Note that $\B$ correspond to
\beq\label{abs}
\B_{\rm abs}(\omega)=0\hspace{20pt}{\rm if~ and~ only~ if}\hspace{20pt}\left\{\begin{array}{l}\omega_{\rm norm}|_{S^n_{l\sin\alpha}}=0,\\
(d\omega)_{\rm norm}|_{S^n_{l\sin\alpha}}=0,\\
       \end{array}
\right.
\eeq
\beq\label{rel}
\B_{\rm rel}(\omega)=0\hspace{20pt}{\rm if~ and~ only~ if}\hspace{20pt}\left\{\begin{array}{l}\omega_{\rm tan}|_{S^n_{l\sin\alpha}}=0,\\
(d^\dagger\omega)_{\rm tan}|_{S^n_{l\sin\alpha}}=0,\\
       \end{array}
\right.
\eeq

\section{The spectrum of the Laplacian on forms}
\label{s2}

In this section we give the spectrum of the Laplacian on forms. The result for $n=1$, and $n=2$ is in \cite{HMS}, Lemmas 3, and 4. Thus we just need to study the case of $n=3$. Decomposing with respect to the projections on the eigenspaces of the restriction of the Laplacian on the section of the cone (i.e with respect to the angular momenta), the definition of an appropriate self adjoint extension of the Laplace operator (on functions) on a cone reduces to the analysis of the boundary values of a singular Sturm Liouville ordinary second order differential equation on the line segment $(0,l]$. The problem was addressed already by Rellich in \cite{Rel}, who parameterized the self adjoint extensions. In particular, it turns out that there are not boundary values (at zero) for the non zero mode of the angular momentum, while a boundary condition is necessary for the zero modes, and the unique self adjoint extension defined by this boundary condition is the maximal extension, corresponding to the Friedrich extension (see \cite{BS2} or \cite{Che2} for the boundary condition). The same argument  works for the Laplacian on forms. However, in the present situation we do not actually need boundary conditions (at zero)  for forms of positive degree, since the middle homology of the section of the cone is trivial (compare with \cite{Che0}).
Since the eigenvalues for relative boundary conditions follows by Hodge duality, we just give the eigenvalues for absolute boundary conditions. 
In the following, we denote by $\{k:\lambda\}$ the set of eigenvalues $\lambda$
with multiplicity $k$.

\begin{lem}\label{eig1} The spectrum of the (Friedrich extension of the) Laplacian operator $\Delta_{C_\alpha S^1_{l\sin\alpha}}^{(q)}$ on $q$-forms with absolute boundary conditions is (where $\nu={\rm csc}\alpha$):
\begin{align*}
\Sp \Delta_{C_\alpha S^1_{l\sin\alpha}}^{(0)}=& \left\{j_{1,k}^2/l^{2}\right\}_{k=1}^{\infty}\cup \left\{2:(j_{\nu n,k}')^2/l^{2}\right\}_{n,k=1}^\infty, \\
\Sp \Delta_{C_\alpha S^1_{l\sin\alpha}}^{(1)}=& \left\{j_{0,k}^2/l^{2}\right\}_{k=1}^{\infty}\cup\left\{j_{1,k}^2/l^{2}\right\}_{k=1}^\infty\cup
\left\{2:j_{\nu n,k}^2/l^{2}\right\}_{n,k=1}^\infty \\ & \cup \left\{2:(j_{\nu n,k}')^2/l^{2}\right\}_{n,k=1}^\infty , \\
\Sp \Delta_{C_\alpha S^1_{l\sin\alpha}}^{(2)}=& \left\{j_{0,k}^2/l^{2}\right\}_{k=1}^\infty\cup \left\{2:j_{\nu n,k}^2/l^{2}\right\}_{n,k=1}^\infty. \\
\end{align*}

\end{lem}

\begin{lem}\label{eig2}
The spectrum of the (Friedrich extension of the) Laplacian operator $\Delta_{C_{\alpha} S^2_{l\sin\alpha}}^{(q)}$ on $q$-forms with absolute boundary conditions is:
\begin{align*}
\Sp \Delta_{C_{\alpha} S^2_{l\sin\alpha}}^{(0)}=& \left\{(2n+1): \hat j_{\mu_n,k,-}^2/l^{2}\right\}_{n,k=1}^{\infty}
\cup \left\{j_{\frac{3}{2},k}^2/l^{2}\right\}_{k=1}^\infty, \\
\Sp \Delta_{C_{\alpha} S^2_{l\sin\alpha}}^{(1)}=& \left\{j_{\frac{3}{2},k}^2/l^{2}\right\}_{k=1}^\infty\cup
\left\{(2n+1):j_{\mu_n, k}^2/l^{2}\right\}_{n,k=1}^\infty\\
&\cup\left\{(2n+1):\hat j_{\mu_n,k,+}^2/l^{2}\right\}_{n,k=1}^\infty\cup
\left\{(2n+1):\hat j_{\mu_n,k,-}^2/l^{2}\right\}_{n,k=1}^\infty,\\
\Sp \Delta_{C_{\alpha} S^2_{l\sin\alpha}}^{(2)}=& \left\{j_{\frac{1}{2},k}^2/l^{2}\right\}_{k=1}^\infty\cup
\left\{(2n+1):j_{\mu_n, k}^2/l^{2}\right\}_{n,k=1}^\infty\\
&\cup\left\{(2n+1):\hat j_{\mu_n,k,+}^2/l^{2}\right\}_{n,k=1}^\infty\cup
\left\{(2n+1):j_{\mu_n,k}^2/l^{2}\right\}_{n,k=1}^\infty, \\
\Sp \Delta_{C_{\alpha} S^2_{l\sin\alpha}}^{(3)}=& \left\{(2n+1):j_{\mu_n,k}^2/l^{2}\right\}_{n,k=1}^{\infty} \cup
\left\{j_{\frac{1}{2},k}^2/l^{2}\right\}_{k=1}^\infty ,\\
\end{align*}
where $\mu_n=\sqrt{\nu^2 n(n+1)+\frac{1}{4}}$, and where the $\hat j_{\nu,k,\pm}$ are the zeros of the function $G^{\pm}_{\nu}(z)=\pm\frac{1}{2}J_{\nu}(z)+zJ'_\nu(z)$.
\end{lem}

\begin{lem}\label{eig3}
The spectrum of the (Friedrich extension of the) Laplacian operator $\Delta_{C_{\alpha} S^3_{l\sin\alpha}}^{(q)}$ on
$q$-forms with absolute boundary conditions is:
\begin{align*}
\Sp \Delta_{C_{\alpha} S^3_{l\sin\alpha}}^{(0)}=& \left\{j_{2,k}^2/l^{2}\right\}_{k=1}^\infty \cup \left\{(n+1)^2:
\tilde j_{\mu_{0,n},k,-}^2/l^{2}\right\}_{n,k=1}^{\infty}, \\
\Sp \Delta_{C_{\alpha} S^3_{l\sin\alpha}}^{(1)}=& \left\{j_{2}^2/l^{2}\right\}_{k=1}^\infty\cup
\left\{2n(n+2):(j'_{\mu_{1,n}, k})^2/l^{2}\right\}_{n,k=1}^\infty\\
&\cup\left\{(n+1)^2:\tilde j_{\mu_{0,n},k,-}^2/l^{2}\right\}_{n,k=1}^\infty\cup
\left\{(n+1)^2: j_{\mu_{0,n},k}^2/l^{2}\right\}_{n,k=1}^\infty,\\
\Sp \Delta_{C_{\alpha} S^3_{l\sin\alpha}}^{(2)}=& \left\{(n+1)^2: \tilde
j_{\mu_{0,n},k,+}^2/l^{2}\right\}_{n,k=1}^\infty \cup
\left\{2n(n+2):(j'_{\mu_{1,n}, k})^2/l^{2}\right\}_{n,k=1}^\infty\\
&\cup\left\{2n(n+2): j_{\mu_{1,n},k}^2/l^{2}\right\}_{n,k=1}^\infty\cup
\left\{(n+1)^2: j_{\mu_{0,n},k}^2/l^{2}\right\}_{n,k=1}^\infty, \\
\Sp \Delta_{C_{\alpha} S^3_{l\sin\alpha}}^{(3)}=& \left\{j_{1}^2/l^{2}\right\}_{k=1}^\infty\cup
\left\{(n+1)^2:\tilde j_{\mu_{0,n}, k,+}^2/l^{2}\right\}_{n,k=1}^\infty\\
&\cup\left\{(n+1)^2: j_{\mu_{0,n},k}^2/l^{2}\right\}_{n,k=1}^\infty\cup
\left\{2n(n+2): j_{\mu_{1,n},k}^2/l^{2}\right\}_{n,k=1}^\infty,\\
\Sp \Delta_{C_{\alpha} S^3_{l\sin\alpha}}^{(4)}=& \left\{j_{1,k}^2/l^{2}\right\}_{k=1}^\infty \cup \left\{(n+1)^2:
\tilde j_{\mu_{0,n},k}^2/l^{2}\right\}_{n,k=1}^{\infty},
\end{align*}
where 
\[
\mu_{0,n}=\sqrt{\nu^2 n(n+2)+1}, \qquad \mu_{1,n} = \nu(n+1),
\]
and where the $\tilde j_{\nu,k,\pm}$ are the zeros of the function $T^{\pm}_{\nu}(z)=\pm J_{\nu}(z)+zJ'_\nu(z)$.
\end{lem}

\begin{proof} Recall we parameterize $C_{\alpha}S^{3}_{l\sin\alpha}$ by
\[C_{\alpha}S^2_{l\sin\alpha}=\begin{cases}
x_1=x\sin{\alpha}\sin{\theta_3}\sin{\theta_2}\cos{\theta_1} \\[8pt]
x_2=x\sin{\alpha}\sin{\theta_3}\sin{\theta_2}\sin{\theta_1} \\[8pt]
x_3=x\sin{\alpha}\sin{\theta_3}\cos{\theta_2} \\[8pt]
x_4=x\sin{\alpha}\cos{\theta_3} \\[8pt]
x_5=x\cos{\alpha}
\end{cases}\] where $(x,\theta_1,\theta_2,\theta_{3})\in [0,l]\times[0,2\pi]\times[0,\pi]\times[0,\pi]$, $0<\alpha\leq \pi/2$ is
a fixed real number and $0< a=\sin\alpha\leq 1$. The induced metric is (for $x>0$)
\[
g = dx\otimes dx + (a^2x^2\sin^2\theta_2\sin^2\theta_3) d\theta_1\otimes d\theta_1 + (a^2x^2\sin^2\theta_3) d\theta_2 \otimes
     d\theta_2 + (a^2x^2) d\theta_3 \otimes d\theta_3.
\]

Using the absolute boundary conditions on forms described in equation (\ref{abs}) 
of the previous section, we obtain the following equations. For the $0$-forms: 
\beq\label{abs0S3}
\begin{aligned}
{\rm abs. }&: \partial_x \omega(l,\theta_1,\theta_2,\theta_3)=0.
\end{aligned}
\eeq

For the $1$-forms: \beq\label{abs1S3}
\begin{aligned}
{\rm abs.}&:
    \left\{\begin{array}{lll} \omega_x (l,\theta_1,\theta_2,\theta_3) = 0 \\
    \partial_x \omega_{\theta_1}(l,\theta_1,\theta_2,\theta_3)=0 \\
    \partial_x\omega_{\theta_2}(l,\theta_1,\theta_2,\theta_3)=0 \\
    \partial_x\omega_{\theta_3}(l,\theta_1,\theta_2,\theta_3)=0.\end{array}\right.
\end{aligned}
\eeq

For the $2$-forms, with $i=1,2,3$: \beq\label{abs2S3}
\begin{aligned}
{\rm abs.}&:
    \left\{\begin{array}{ll} \omega_{x\theta_i} (l,\theta_1,\theta_2,\theta_3) = 0 \\
    \partial_x\omega_{\theta_1\theta_2}(l,\theta_1,\theta_2,\theta_3)=0 \\
    \partial_x\omega_{\theta_1\theta_3}(l,\theta_1,\theta_2,\theta_3)=0 \\
    \partial_x\omega_{\theta_2\theta_3}(l,\theta_1,\theta_2,\theta_3)=0.\end{array}\right.
\end{aligned}
\eeq

For the $3$-forms: \beq\label{abs3S3}
\begin{aligned}
{\rm abs.}&:
    \left\{\begin{array}{lll} \omega_{x\theta_1\theta_2} (l,\theta_1,\theta_2,\theta_3) = 0 \\
    \omega_{x\theta_1\theta_3}(l,\theta_1,\theta_2,\theta_3)=0 \\
    \omega_{x\theta_2\theta_3}(l,\theta_1,\theta_2,\theta_3)=0 \\
    \partial_x\omega_{\theta_1\theta_2\theta_3}(l,\theta_1,\theta_2,\theta_3)=0.\end{array}\right.
\end{aligned}
\eeq

For the $4$-forms: \beq\label{abs4S3}
\begin{aligned}
{\rm abs. }&: \omega_{x\theta_1\theta_2\theta_3}(l,\theta_1,\theta_2,\theta_3)=0.
\end{aligned}
\eeq

Next we use the description of the eigenfunctions given in Section 3 of \cite{Che2} to determine the eigenvalues. By \cite{IT} the eigenvalues of the coexact forms of the Laplacian over $S^3$ are, with $n\geq 1$:

\begin{center}

\begin{table}[htb]
\centering
\begin{tabular}{|c|c|c|}
\hline Dimension & Eigenvalue & Multiplicity   \\
\hline \hline
$0$ & $n(n+2)$ & $(n+1)^2$  \\
\cline{1-3}
$1$ & $(n+1)^2$ & $2 n(n+2)$   \\
\cline{1-3} $2$ & $n(n+2)$ &  $(n+1)^2$\\
\hline
\end{tabular}
\end{table}

\end{center}

And by \cite{Che2} we have $\mu_{0,n} = \mu_{2,n} = \sqrt{\nu^2 n(n+2) +1 }$ and $\mu_{1,n} = \nu(n+1)$, and the
eigenforms of the Laplacian of $C_{\alpha}S^{3}_{la}$ are as follows. For the $0$-forms:
\begin{equation*}
\alpha^{(0)}_{n} = x^{-1} J_{\mu_{0,n}}(\lambda x) \phi^{(0)}_{n}(\theta_1,\theta_2,\theta_3), \qquad E^{(0)} = x^{-1}
J_{1}(\lambda x) h^{0}(\theta_1,\theta_2,\theta_3).
\end{equation*}

For the $1$-forms:
\begin{align*}
\alpha^{(1)}_{n} &= x^{-1} J_{\mu_{1,n}}(\lambda x) \phi^{(1)}_{n}(\theta_1,\theta_2,\theta_3),\\
\beta^{ (1)}_{n} &= x^{-1} J_{\mu_{0,n}}(\lambda x)d\phi^{(0)}_{n}(\theta_1,\theta_2,\theta_3) + \partial_x(x^{-1}
J_{\mu_{0,n}}(\lambda x))dx\wedge
\phi^{(0)}_{n}(\theta_1,\theta_2,\theta_3),\\
\gamma^{ (1)}_{n} &= x^{-1} \partial_x(x J_{\mu_{0,n}}(\lambda x)) d\phi^{(0)}_{n}(\theta_1,\theta_2,\theta_3) +x^{-2}
J_{\mu_{0,n}}(\lambda x) dx\wedge \tilde \delta \tilde d \phi^{(0)}_{n}(\theta_1,\theta_2,\theta_3),\\
D^{(1)} &= \partial_x(x^{-1} J_1(\lambda x)) dx \wedge h^{(0)}(\theta_1,\theta_2,\theta_3)
\end{align*}

For the $2$-forms:
\begin{align*}
\alpha^{(2)}_{n} &= x J_{\mu_{0,n}}(\lambda x) \phi^{(2)}_{n}(\theta_1,\theta_2,\theta_3),\\
\beta^{ (2)}_{n} &=  J_{\mu_{1,n}}(\lambda x)d\phi^{(1)}_{n}(\theta_1,\theta_2,\theta_3) +
\partial_x(J_{\mu_{1,n}}(\lambda x))dx\wedge\phi^{(2)}_{n}(\theta_1,\theta_2,\theta_3),\\
\gamma^{ (2)}_{n} &= x \partial_x(J_{\mu_{1,n}}(\lambda x)) d\phi^{(1)}_{n}(\theta_1,\theta_2,\theta_3) +x^{-1}
J_{\mu_{1,n}}(\lambda x) dx\wedge \tilde \delta \tilde d \phi^{(1)}_{n}(\theta_1,\theta_2,\theta_3),\\
\delta^{(2)}_n &= J_{\mu_{0,n}}(\lambda x) dx \wedge d\phi^{(0)}_{n}(\theta_1,\theta_2,\theta_3).
\end{align*}

For the $3$-forms:
\begin{align*}
\beta^{ (3)}_{n} &= x J_{\mu_{0,n}}(\lambda x)d\phi^{(2)}_{n}(\theta_1,\theta_2,\theta_3) +
\partial_x(x J_{\mu_{0,n}}(\lambda x))dx\wedge\phi^{(2)}_{n}(\theta_1,\theta_2,\theta_3),\\
\gamma^{ (3)}_{n} &= x^{3} \partial_x(x^{-1} J_{\mu_{0,n}}(\lambda x)) d\phi^{(2)}_{n}(\theta_1,\theta_2,\theta_3) +
J_{\mu_{0,n}}(\lambda x) dx\wedge \tilde \delta \tilde d \phi^{(2)}_{n}(\theta_1,\theta_2,\theta_3),\\
\delta^{(3)}_n &= x J_{\mu_{1,n}}(\lambda x) dx \wedge d\phi^{(1)}_{n}(\theta_1,\theta_2,\theta_3),\\
E^{(3)} & = x^2 J_{2}(\lambda x)h^{3}(\theta_1,\theta_2,\theta_3).
\end{align*}

For $4$-forms:
\begin{equation*}
\delta^{(4)}_{n} = x^{2} J_{\mu_{0,n}}(\lambda x) dx \wedge d\phi^{(2)}_{n}(\theta_1,\theta_2,\theta_3), \qquad D^{(4)}
=\partial_x( x^{2} J_{2}(\lambda x)) dx \wedge h^{3}(\theta_1,\theta_2,\theta_3).
\end{equation*}

Where the $\phi^{(i)}_n(\theta_1,\theta_2,\theta_3)$, for $i=0,1,2$,  are coexact eigenforms of the Laplacian on $S^{3}$, and $h^{(0)}(\theta_1,\theta_2,\theta_3)$, and $h^{(3)}(\theta_1,\theta_2,\theta_3)$ are harmonic forms of the Laplacian on  $S^{3}$. Using these functions in the boundary conditions given in equation (\ref{abs}), we obtain the result.

\end{proof}

\section{Zeta determinants for some class of double sequences}
\label{s3}

We give in this section all the tools necessary in order to evaluate the zeta determinants appearing in the calculation of the analytic torsion. This is based on \cite{Spr3} \cite{Spr4} \cite{Spr5} and \cite{Spr9}. We present here a simplified version of the main result of those works (see in particular the  general formulation in Theorem 3.9 of \cite{Spr9} or the Spectral Decomposition Lemma of \cite{Spr5}), that is sufficient for our purpose here.

Let $S=\{a_n\}_{n=1}^\infty$ be a sequence of non vanishing complex numbers, ordered by increasing modules, with the unique point of accumulation at infinite. The positive real number (possibly infinite)
\[
s_0={\rm limsup}_{n\to\infty} \frac{\log n}{\log |a_n|},
\]
is called the exponent of convergence of $S$, and denoted by $\ec(S)$. We are only interested in sequences with  $\ec(S)=s_0<\infty$. If this is the case, then there exists a least integer $p$ such that the series $\sum_{n=1}^\infty a_n^{-p-1}$ converges absolutely. We assume $s_0-1< p\leq s_0$, we call the integer $p$ the genus of the sequence $S$, and we write $p=\ge(S)$.  We define the zeta function associated to $S$ by the uniformly convergent series
\[
\zeta(s,S)=\sum_{n=1}^\infty a_n^{-s},
\]
when $\Re(s)> \ec(S)$, and by analytic continuation otherwise.  We call the open subset $\rho(S)=\C-S$ of the complex plane the  resolvent set of $S$. For all $\lambda\in\rho(S)$, we define the Gamma function associated to $S$  by the canonical product
\beq\label{gamma}
\frac{1}{\Gamma(-\lambda,S)}=\prod_{n=1}^\infty\left(1+\frac{-\lambda}{a_n}\right)\e^{\sum_{j=1}^{\ge(S)}\frac{(-1)^j}{j}\frac{(-\lambda)^j}{a_n^j}}.
\eeq

When necessary in order to define the meromorphic branch of an analytic function, the domain for $\lambda$ will be the  open subset $\C-[0,\infty)$ of the complex plane.
We use the notation $\Sigma_{\theta,c}=\left\{z\in \C~|~|\arg(z-c)|\leq \frac{\theta}{2}\right\}$,
with $c\geq \delta> 0$, $0< \theta<\pi$. We use
$D_{\theta,c}=\C-\Sigma_{\theta,c}$, for the complementary (open) domain and $\Lambda_{\theta,c}=\partial \Sigma_{\theta,c}=\left\{z\in \C~|~|\arg(z-c)|= \frac{\theta}{2}\right\}$, oriented counter clockwise, for the boundary.
With this notation, we define now a particular subclass of sequences. Let $S$ be as above, and assume that $\ec(S)<\infty$, and that there exist $c>0$ and $0<\theta<\pi$, such that $S$ is contained in the interior of the sector $\Sigma_{\theta,c}$. Furthermore, assume that the logarithm of the associated Gamma function has a uniform asymptotic expansion for large $\lambda\in D_{\theta,c}(S)=\C-\Sigma_{\theta,c}$ of the following form
\[
\log\Gamma(-\lambda,S)\sim\sum_{j=0}^\infty a_{\alpha_j,0}(-\lambda)^{\alpha_j} +\sum_{k=0}^{\ge(S)} a_{k,1}(-\lambda)^k\log(-\lambda),
\]
where $\{\alpha_j\}$ is a decreasing sequence of real numbers. Then, we say that $S$ is a {\it totally regular sequence of spectral type with infinite order}. We call the open set $D_{\theta,c}(S)$ the asymptotic domain of $S$.

Next, let $S=\{\lambda_{n,k}\}_{n,k=1}^\infty$ be a double sequence of non
vanishing complex numbers with unique accumulation point at the
infinity, finite exponent $s_0=\ec(S)$ and genus $p=\ge(S)$. Assume if necessary that the elements of $S$ are ordered as $0<|\lambda_{1,1}|\leq|\lambda_{1,2}|\leq |\lambda_{2,1}|\leq \dots$. We use the notation $S_n$ ($S_k$) to denote the simple sequence with fixed $n$ ($k$). We call the exponents of $S_n$ and $S_k$ the relative exponents of $S$, and we use the notation $(s_0=\ec(S),s_1=\ec(S_k),s_2=\ec(S_n))$. We define relative genus accordingly.

\begin{defi} Let $S=\{\lambda_{n,k}\}_{n,k=1}^\infty$ be a double
sequence with finite exponents $(s_0,s_1,s_2)$, genus
$(p_0,p_1,p_2)$, and positive spectral sector
$\Sigma_{\theta_0,c_0}$. Let $U=\{u_n\}_{n=1}^\infty$ be a totally
regular sequence of spectral type of infinite order with exponent
$r_0$, genus $q$, domain $D_{\phi,d}$. We say that $S$ is
spectrally decomposable over $U$ with power $\kappa$, length $\ell$ and
asymptotic domain $D_{\theta,c}$, with $c={\rm min}(c_0,d,c')$,
$\theta={\rm max}(\theta_0,\phi,\theta')$, if there exist positive
real numbers $\kappa$, $\ell$ (integer), $c'$, and $\theta'$, with
$0< \theta'<\pi$,   such that:
\begin{enumerate}
\item the sequence
$u_n^{-\kappa}S_n=\left\{\frac{\lambda_{n,k}}{u^\kappa_n}\right\}_{k=1}^\infty$ has
spectral sector $\Sigma_{\theta',c'}$, and is a totally regular
sequence of spectral type of infinite order for each $n$;
\item the logarithmic $\Gamma$-function associated to  $S_n/u_n^\kappa$ has an asymptotic expansion  for large
$n$ uniformly in $\lambda$ for $\lambda$ in
$D_{\theta,c}$, of the following form
\beq\label{exp}
\hspace{30pt}\log\Gamma(-\lambda,u_n^{-\kappa} S_n)=\sum_{h=0}^{\ell}
\phi_{\sigma_h}(\lambda) u_n^{-\sigma_h}+\sum_{l=0}^{L}
P_{\rho_l}(\lambda) u_n^{-\rho_l}\log u_n+o(u_n^{-r_0}),
\eeq
where $\sigma_h$ and $\rho_l$ are real numbers with $\sigma_0<\dots <\sigma_\ell$, $\rho_0<\dots <\rho_L$, the
$P_{\rho_l}(\lambda)$ are polynomials in $\lambda$ satisfying the condition $P_{\rho_l}(0)=0$, $\ell$ and $L$ are the larger integers
such that $\sigma_\ell\leq r_0$ and $\rho_L\leq r_0$.
\end{enumerate}
\label{spdec}
\end{defi}

When a double sequence $S$ is spectrally decomposable over a simple sequence $U$, Theorem 3.9 of \cite{Spr9} gives a formula for the derivative of the associated zeta function at zero. In order to understand such a formula, we need to introduce some other quantities. First, we define the functions
\beq\label{fi1}
\Phi_{\sigma_h}(s)=\int_0^\infty t^{s-1}\frac{1}{2\pi i}\int_{\Lambda_{\theta,c}}\frac{\e^{-\lambda t}}{-\lambda} \phi_{\sigma_h}(\lambda) d\lambda dt.
\eeq

Next, by Lemma 3.3 of \cite{Spr9}, for all $n$, we have the expansions:
\beq\label{form}\begin{aligned}
\log\Gamma(-\lambda,S_n/{u_n^\kappa})&\sim\sum_{j=0}^\infty a_{\alpha_j,0,n}
(-\lambda)^{\alpha_j}+\sum_{k=0}^{p_2} a_{k,1,n}(-\lambda)^k\log(-\lambda),\\
\phi_{\sigma_h}(\lambda)&\sim\sum_{j=0}^\infty b_{\sigma_h,\alpha_j,0}
(-\lambda)^{\alpha_j}+\sum_{k=0}^{p_2} b_{\sigma_h,k,1}(-\lambda)^k\log(-\lambda),
\end{aligned}
\eeq
for large $\lambda$ in $D_{\theta,c}$. We set (see Lemma 3.5 of \cite{Spr9})
\beq\label{fi2}
\begin{aligned}
A_{0,0}(s)&=\sum_{n=1}^\infty \left(a_{0, 0,n} -\sum_{h=0}^\ell
b_{\sigma_h,0,0}u_n^{-\sigma_h}\right)u_n^{-\kappa s},\\
A_{j,1}(s)&=\sum_{n=1}^\infty \left(a_{j, 1,n} -\sum_{h=0}^\ell
b_{\sigma_h,j,1}u_n^{-\sigma_h}\right)u_n^{-\kappa s},
~~~0\leq j\leq p_2.
\end{aligned}
\eeq

We can now state the formula for the derivative at zero of the double zeta function. We give here a modified version of Theorem 3.9 of \cite{Spr9}, more suitable for our purpose here. This is based on the following fact. The key point in the proof of Theorem 3.9 of \cite{Spr9} is the decomposition given in Lemma 3.5 of that paper of the sum
\[
\mathcal{T}(s,\lambda, S,U)=\sum_{n=1}^\infty u_n^{-\kappa s} \log\Gamma(-\lambda, u_n^{-\kappa}S_n),
\]
in two terms: the regular part $\mathcal{P}(s,\lambda,S,U)$ and the remaining singular part. The regular part is obtained subtracting from $\T$ some terms constructed starting from the expansion of the logarithmic Gamma function given in equation (\ref{exp}), namely
\[
\P(s,\lambda,S,u)=\T(s,\lambda, S,U)-\sum_{h=0}^{\ell}
\phi_{\sigma_h}(\lambda) u_n^{-\sigma_h}-\sum_{l=0}^{L}
P_{\rho_l}(\lambda)u_n^{-\rho_l}\log u_n.
\]

Now, assume instead we subtract only the terms such that the zeta function $\zeta(s,U)$ has a pole at $s=\sigma_h$ or at $s=\rho_l$. Let $\hat \P(s,\lambda, S,U)$ be the resulting function. Then the same argument as the one used in Section 3 of \cite{Spr9} in order to prove Theorem 3.9 applies, and we obtain similar formulas for the values of the residue, and of the finite part of the zeta function $\zeta(s,S)$ and of its derivative at zero, with just two differences: first, in the all the sums, all the terms with index $\sigma_h$ such that $s=\sigma_h$ is not a pole of $\zeta(s,U)$ must be omitted; and second, we must substitute the terms $A_{0,0}(0)$ and $A_{0,1}'(0)$, with the finite parts $\Rz_{s=0}A_{0,0}(s)$, and $\Rz_{s=0}A_{0,1}'(s)$. The first modification is an obvious consequence of the substitution of the function $\P$ by the function $\hat \P$. The second modification, follows by the same reason noting that the function $A_{\alpha_j,k}(s)$ defined in Lemma 3.5 of \cite{Spr9} are no longer regular at $s=0$ themselves. However, they both admits a meromorphic extension regular at $s=0$, using the extension of the zeta function $\zeta(s,U)$, and the expansion of the coefficients $a_{\alpha_j,k,n}$ for large $n$.
Thus we have the following result.

\begin{theo} \label{tt} The formulas of Theorem 3.9 of \cite{Spr9} hold if all the quantities with index $\sigma_h$ such that the zeta function $\zeta(s,U)$ has not a pole at $s=\sigma_h$ are omitted. In such a case, the result must be read by means of the analytic extension of the zeta function $\zeta(s,U)$.
\end{theo}

Next, assuming some simplified pole structure for the zeta function $\zeta(s,U)$, sufficient for the present analysis, we state the main result of this section.

\begin{theo} \label{t4} Let $S$ be spectrally decomposable over $U$ as in Definition \ref{spdec}. Assume that the functions $\Phi_{\sigma_h}(s)$ have at most simple poles for $s=0$. Then,
$\zeta(s,S)$ is regular at $s=0$, and
\begin{align*}
\zeta(0,S)=&-A_{0,1}(0)+\frac{1}{\kappa}{\sum_{h=0}^\ell} \Ru_{s=0}\Phi_{\sigma_h}(s)\Ru_{s=\sigma_h}\zeta(s,U),\\
\zeta'(0,S)=&-A_{0,0}(0)-A_{0,1}'(0)+\frac{\gamma}{\kappa}\sum_{h=0}^\ell\Ru_{s=0}\Phi_{\sigma_h}(s)\Ru_{s=\sigma_h}\zeta(s,U)\\
&+\frac{1}{\kappa}\sum_{h=0}^\ell\Rz_{s=0}\Phi_{\sigma_h}(s)\Ru_{s=\sigma_h}\zeta(s,U)+{\sum_{h=0}^\ell}{^{\displaystyle
'}}\Ru_{s=0}\Phi_{\sigma_h}(s)\Rz_{s=\sigma_h}\zeta(s,U),
\end{align*}
where the notation $\sum'$ means that only the terms such that $\zeta(s,U)$ has a pole at $s=\sigma_h$ appear in the sum.

\end{theo}

This result should be compared with the Spectral Decomposition Lemma  of \cite{Spr5} and Proposition 1 of \cite{Spr6}.

\begin{corol} \label{c} Let $S_{(j)}=\{\lambda_{(j),n,k}\}_{n,k=1}^\infty$, $j=1,2$, be two double sequences that satisfy all the requirements of Definition \ref{spdec} of spectral decomposability over a common sequence $U$, with the same parameters $\kappa$, $\ell$, etc., except that the polynomials $P_{(j),\rho}(\lambda)$ appearing in condition (2) do not vanish for $\lambda=0$. Assume that the difference of such polynomials does satisfy this condition, namely that $P_{(1),\rho}(0)-P_{(2),\rho}(0)=0$. Then, the difference of the zeta functions $\zeta(s,S_{(1)})-\zeta(s,S_{(2)})$ is regular at $s=0$ and satisfies the formulas given in Theorem \ref{t4}.
\end{corol}

We conclude this section by recalling some results on zeta determinants of some simple sequences that will be necessary in the following. This results can be found in different places, and are known to specialists. We will use the formulation of \cite{Spr1}. For positive real numbers $l$ and $q$, define the non homogeneous quadratic Bessel zeta function by
\[
z(s,\nu,q,l)=\sum_{k=1}^\infty \left(\frac{j_{\nu,k}^2}{l^2}+q^2\right)^{-s},
\]
for $\Re(s)>\frac{1}{2}$. Then, $z(s,\nu,q,l)$ extends analytically to a meromorphic function in the complex plane with simple poles at $s=\frac{1}{2}, -\frac{1}{2}, -\frac{3}{2}, \dots$. The point $s=0$ is a regular point and
\beq\label{p00}
\begin{aligned}
z(0,\nu,q,l)&=-\frac{1}{2}\left(\nu+\frac{1}{2}\right),\\
z'(0,\nu,q,l)&=-\log\sqrt{2\pi l}\frac{I_\nu(lq)}{q^\nu}.
\end{aligned}
\eeq

In particular, taking the limit for $q\to 0$,
\[
z'(0,\nu,0,l)=-\log\frac{\sqrt{\pi}l^{\nu+\frac{1}{2}}}{2^{\nu-\frac{1}{2}}\Gamma(\nu+1)}.
\]

\section{The analytic torsion }
\label{s4}

In this section we give the analytic torsions of $C_\alpha S^n_{l\sin\alpha}$, for $n=1,2$, and $3$. Actually, the case $n=1$ is essentially contained in \cite{Spr6}, and both the cases $n=1$ and $n=2$ are given in \cite{HMS}, Sections 5.4 and 5.5, so we will focus here on the new case of $C_\alpha S^3_{l\sin\alpha}$. 
By the analysis in Section \ref{s3}, the relevant zeta functions are
\begin{align*}
\zeta(s,\Delta^{(1)})&=\sum_{k=1}^\infty \frac{j_{2,k}^{-2s}}{l^{-2s}}+2\sum_{n,k=1}^\infty n(n+2)
\frac{(j'_{\mu_{1,n},k})^{-2s}}{l^{-2s}} + \sum_{n,k=1}^{\infty} (n+1)^2 \frac{\tilde
j_{\mu_{0,n},k,-}^{-2s}}{l^{-2s}}\\
&+\sum_{n,k=1}^{\infty} (n+1)^2 \frac{j_{\mu_{0,n},k}^{-2s}}{l^{-2s}},\\
\zeta(s,\Delta^{(2)})&= \sum_{n,k=1}^{\infty} (n+1)^2 \frac{\tilde j_{\mu_{0,n},k,+}^{-2s}}{l^{-2s}} +
2\sum_{n,k=1}^{\infty} n(n+2) \frac{(j'_{\mu_{1,n},k})^{-2s}}{l^{-2s}} \\
&+2 \sum_{n,k=1}^{\infty} n(n+2) \frac{j_{\mu_{1,n},k}^{-2s}}{l^{-2s}} +
\sum_{n,k=1}^{\infty} (n+1)^2 \frac{j_{\mu_{0,n},k}^{-2s}}{l^{-2s}},\\
\zeta(s,\Delta^{(3)})&=\sum_{k=1}^\infty \frac{j_{1,k}^{-2s}}{l^{-2s}}+2\sum_{n,k=1}^\infty n(n+2)
\frac{j_{\mu_{1,n},k}^{-2s}}{l^{-2s}} + \sum_{n,k=1}^{\infty} (n+1)^2 \frac{\tilde j_{\mu_{0,n},k,+}^{-2s}}{l^{-2s}}\\
&+ \sum_{n,k=1}^{\infty} (n+1)^2 \frac{j_{\mu_{0,n},k}^{-2s}}{l^{-2s}},\\
\zeta(s,\Delta^{(4)})&=\sum_{k=1}^{\infty} \frac{j_{1,k}^{-2s}}{l^{-2s}} + \sum_{n,k=1}^{\infty} (n+1)^2
\frac{j_{\mu_{0,n},k}^{-2s}}{l^{-2s}},
\end{align*}
and by equation (\ref{analytic}), the torsion is ($a=\sin\alpha=\frac{1}{\nu}$)
\begin{align*}
\log T(C_\alpha S^3_{la})&=-\frac{1}{2}\zeta'(0,\Delta^{(1)}) +\zeta'(0,\Delta^{(2)})-\frac{3}{2}\zeta'(0,\Delta^{(3)})+2\zeta'(0,\Delta^{(4)}).
\end{align*}

Define the function
\begin{align*}
t(s)=&-\frac{1}{2}\zeta(s,\Delta^{(1)}) +\zeta(s,\Delta^{(2)})-\frac{3}{2}\zeta(s,\Delta^{(3)})+2\zeta(s,\Delta^{(4)})\\
=&\frac{1}{2} \sum_{k=1}^{\infty} \frac{j_{1,k}^{-2s}}{l^{-2s}}-\frac{1}{2} \sum_{k=1}^\infty
\frac{j_{2,k}^{-2s}}{l^{-2s}}\\
& +  \sum_{n,k=1}^\infty n(n+2) \frac{(j'_{\mu_{1,n},k})^{-2s}}{l^{-2s}} -
\sum_{n,k=1}^{\infty} n(n+2) \frac{j_{\mu_{1,n},k}^{-2s}}{l^{-2s}}\\
&+ \sum_{n,k=1}^{\infty} (n+1)^2 \frac{j_{\mu_{0,n},k}^{-2s}}{l^{-2s}} -\frac{1}{2}\sum_{n,k=1}^{\infty} (n+1)^2
\frac{\tilde j_{\mu_{0,n},k,-}^{-2s}}{l^{-2s}} -\frac{1}{2} \sum_{n,k=1}^{\infty} (n+1)^2 \frac{\tilde
j_{\mu_{0,n},k,+}^{-2s}}{l^{-2s}}\\
=& l^{2s}\left(\frac{1}{2}z_{1}(s) - \frac{1}{2}z_{2}(s) + \hat{Z}(s) -Z(s) +Z_0(s)
-\frac{1}{2}Z_+(s) - \frac{1}{2}Z_-(s) \right) ,
\end{align*}
then
\begin{align*}
\log T(C_\alpha S^3_{la})=t'(0)=&\frac{1}{2}z'_{1}(0) - \frac{1}{2}z'_{2}(0) + \hat{Z}'(0) -Z'(0) +Z'_0(0)\\
&-\frac{1}{2}Z'_+(0)- \frac{1}{2}Z'_-(0) +
\log l^2 \left(\frac{1}{2}z_{1}(0) - \frac{1}{2}z_{2}(0)\right.\\
& \left.+\hat{Z}(0) -Z(0) +Z_0(0) -\frac{1}{2}Z_+(0) - \frac{1}{2}Z_-(0)\right).
\end{align*}

Using equations (\ref{p00})  of Section \ref{s3}, we compute $z_{1/2}(0)$ e $z'_{1/2}(0)$. We obtain

\begin{equation}\label{ttt}
\begin{aligned}
\log T(C_\alpha S^3_{la})&= \left(\frac{1}{4} +
\hat{Z}(0) -Z(0) +Z_0(0) -\frac{1}{2}Z_+(0) - \frac{1}{2}Z_-(0)\right)\log l^2\\
&+ \left(-\log 2 + \hat{Z}'(0) -Z'(0) +Z'_0(0) -\frac{1}{2}Z'_+(0) - \frac{1}{2}Z'_-(0) \right).
\end{aligned}
\end{equation}

In order to evaluate the remaining part, we use Corollary \ref{c} of Theorem \ref{t4}. We consider separately the two functions $Z(s)-\hat Z(s)$, and $2Z_0(s)-Z_+(s)-Z_-(s)$. In the first case, the relevant sequences are the double sequences $S=\{n(n+2):j_{\mu_{1,n},k}^2\}$ and $\hat S=\{n(n+2):(j_{\mu_{1,n},k}')^2\}$, and the simple sequence $U_1=\{n(n+2):\mu_{1,n}\}_{n=1}^\infty$,  and $Z(s)=\zeta(s,S)$, $\hat
Z(s)=\zeta(s,\hat S)$. In the second case, the relevant sequences are the double sequences $S_0=\{(n+1)^2:j_{\mu_{0,n},k}^2\}$ and $S_\pm=\{(n+1)^2:(\tilde j_{\mu_{0,n,\pm},k})^2\}$, and the simple sequence $U_0=\{\mu_{0,n}\}_{n=1}^\infty$,  and $Z_0(s)=\zeta(s,S_0)$, $Z_\pm(s)=\zeta(s, S_\pm)$. 

We start by analysing the two simple sequences $U_j$, $j=0,1$. Recall from Lemma \ref{eig3}, that
\[
\mu_{0,n}=\sqrt{\nu^2 n(n+2)+1}, \qquad \mu_{1,n} = \nu(n+1).
\]

Consider first the sequence $U_1=\{n(n+2):\mu_{1,n}\}_{n=1}^\infty$. By definition of $\mu_{1,n}$, it is easy to see that
\[
\zeta(s,U_1)=\nu^{-s}\left(\zeta_R(s-2)-\zeta_R(s)\right),
\]
and therefore $U_1$ is a totally regular sequence of spectral type with infinite order, $\ec(U_1)=\ge(U_1)=3$, and  $\zeta(s,U_1)$ has simple poles at $s=1$ and $s=3$ with residues:
\beq\label{rrr1}\begin{aligned}
&\Rz_{s=1}\zeta(s,U_1)=\frac{1}{\nu}\left(\log\nu-\gamma-\frac{1}{12}\right),&\Ru_{s=1}\zeta(s,U_1)=-\frac{1}{\nu},\\
&\Rz_{s=3}\zeta(s,U_1)=\frac{1}{\nu^3}\left(\gamma-\log\nu-\zeta(3)\right),&\Ru_{s=3}\zeta(s,U_1)=\frac{1}{\nu^3}.\\
\end{aligned}
\eeq

The analysis for the sequence $U_0$ is a little bit longer. By definition $U_0=\{(n+1)^2:\mu_{0,n}\}_{n=1}^\infty$, where
\[
\mu_{0,n}=\sqrt{\nu^2 n(n+2)+1}.
\]

For a positive $q$, consider the sequence 
\[
L_q=\{(n+1)^2:\sqrt{n(n+2)+q}\}_{n=1^\infty},
\]

Then, it is clear that 
\[
\zeta(s,U_0)=\nu^{-s}\zeta(s,L_{\frac{1}{\nu^2}}).
\]

The sequence $L_0$ is the sequence of the square roots of the positive eigenvalues of the Laplace operator on the three sphere $S^3$ of radius 1 (see \cite{Spr0}, and references therein). Thus,
\[
\zeta(2s,L_0)=\zeta(s,\Sp_+\Delta_{S^3}).
\]

The zeta function $\zeta(s,\Sp_+\Delta_{S^3})$ has been studied by various author. We will refer to \cite{Spr0}. Using the results in \cite{Spr0}, it follows that $\ec(\Sp_+\Delta^{(0)}_{S^3})=\frac{3}{2}$, $\ge(\Sp_+ \Delta^{(0)}_{S^3})=1$, and that $\Sp_+ \Delta^{(0)}_{S^3}$ is a totally regular sequence of spectral type with infinite order. Since shifting the sequence does not alter its character (see \cite{Spr4}), it follows that $\ec(U_0)=\ge(U_0)=3$, and that $U$ is a totally regular sequence of spectral type with infinite order.
In \cite{Spr0}, it is also proved that $\zeta(s,\Sp_+\Delta_{S^3})$ has simple poles at $s=\frac{3}{2},\frac{1}{2},-\frac{j}{2}$, for all $j>0$, and formulas for the residues are given. In particular:
\begin{align*}
&\Ru_{s=\frac{3}{2}}\zeta(s,\Sp_+ \Delta^{(0)}_{S^3})=\frac{1}{2},&\Ru_{s=3}\zeta(s,\Sp_+ \Delta^{(0)}_{S^3})=\frac{1}{4},\\
\end{align*}
and hence, $\zeta(s,L_0)$ has one simple pole at $s=1$, and $s=3$ with the residues: 
\begin{align*}
&\Ru_{s=3}\zeta(s,\Sp_+ \Delta^{(0)}_{S^3})=1,&\Ru_{s=1}\zeta(s,\Sp_+ \Delta^{(0)}_{S^3})=\frac{1}{2}.\\
\end{align*}

Expanding the power of the binomial, we have that
\begin{align*}
\zeta(s,L_q)&=\zeta(s,L_0)-\frac{s}{2}\zeta(s,L_0)q+\sum_{j=2}^\infty \binom{-\frac{s}{2}}{j}\zeta(s+2j,L_0)q^j,
\end{align*}
and therefore,
\begin{align*}
&\Ru_{s=1}\zeta(s,L_q)=\frac{1}{2}(1-q),&\Ru_{s=3}\zeta(s,L_q)=1,\\
\end{align*}
and we have the expansions
\beq\label{resz1}
\begin{aligned}
\zeta(s,U_0)&=\nu^{-s}\zeta(s,L_q)=\frac{1}{2\nu}\left(1-\frac{1}{\nu^2}\right)\frac{1}{s-1}+K_1(s),& {\rm near}~s=1,\\
\zeta(s,U_0)&=\nu^{-s}\zeta(s,L_q)=\frac{1}{\nu^3}\frac{1}{s-3}+K_3(s),& {\rm near}~s=3,\\
\end{aligned}
\eeq
where the $K_j(s)$ are some regular functions.

Next, we start the analysis of the double sequences. We split it into two parts.

\subsection{Part I} In this first part we deal with $Z(s)-\hat Z(s)$. Thus, we consider the sequences $S$ and $\hat S$. Using classical estimates for the zeros of Bessel function \cite{Wat}, we find  that $\ec(S)=\ec(\hat S)=2$, and the relative genus are $(2,0,1)$ for both sequences. The fact that $S_n$ and $\hat S_n$ are totally regular sequences of spectral type with infinite order, will be a consequence of the following analysis. Note that
we have the product representations (the first is classical, see for example \cite{Wat}, the second follows
using the Hadamard factorization theorem)

\begin{align*}
I_\nu(z)&=\frac{z^\nu}{2^\nu\Gamma(\nu+1)}\prod_{k=1}^\infty \left(1+\frac{z^2}{j_{\nu,k}^2}\right),\\
I_\nu'(z)&=\frac{z^{\nu-1}}{2^\nu\Gamma(\nu)}\prod_{k=1}^\infty \left(1+\frac{z^2}{(j_{\nu,k}')^2}\right).\\
\end{align*}

Using these representations, we obtain the following representations for the Gamma functions associated to the
sequences $S_n$ and $\hat S_n$. For further use, we give instead the representations for the Gamma functions associated
to the sequences $S_n/\mu_{1,n}^2$, and $\hat S_n/\mu_{1,n}^2$, that will do as well. By the definition in equation
(\ref{gamma}), with $z=\sqrt{-\lambda}$, we have
\begin{align*}
\log \Gamma(-\lambda,S_n/(\mu_{1,n})^2)=&-\log\prod_{k=1}^\infty \left(1+\frac{(-\lambda)(\mu_{1,n})^2}{j_{\mu_{1,n},k}^2}\right)\\
=&-\log I_{\mu_{1,n}}(\mu_{1,n}\sqrt{-\lambda})+(\mu_{1,n})\log\sqrt{-\lambda} \\
&+\mu_{1,n}\log (\mu_{1,n})-\mu_{1,n}\log 2-\log\Gamma(\mu_{1,n}+1),\\
\log \Gamma(-\lambda,\hat S_n/(\mu_{1,n})^2)=&-\log\prod_{k=1}^\infty \left(1+\frac{(-\lambda)(\mu_{1,n})^2}{(j_{\mu_{1,n},k}')^2}\right)\\
=&-\log I'_{\mu_{1,n}}(\mu_{1,n}\sqrt{-\lambda})+(\mu_{1,n}-1)\log\sqrt{-\lambda} \\
&+\mu_{1,n}\log (\mu_{1,n})-\mu_{1,n}\log 2-\log\Gamma(\mu_{1,n}).
\end{align*}

A first consequence of these representations is that we have a complete asymptotic expansion of the Gamma functions
$\log \Gamma(-\lambda,S_n)$, and $\log \Gamma(-\lambda,\hat S_n)$, and therefore $S_n$ and $\hat S_n$ are sequences of
spectral type. Considering the expansions, it follows that they are both totally regular sequences of infinite order.

Next, we prove that $S$ and $\hat S$ are spectrally decomposable over $U_1$ with power $\kappa=2$ and length
$\ell=4$, as in Definition \ref{spdec}. We have to show that the functions $\log \Gamma(-\lambda,S_n/\mu_{1,n}^2)$,
and $\log \Gamma(-\lambda,\hat S_n/\mu_{1,n}^2)$ have the appropriate uniform expansions for large $n$. This follows
using the uniform expansions for the Bessel functions given for example in \cite{Olv} (7.18), and Ex. 7.2,
\[
I_{\nu}(\nu z)=\frac{\e^{\nu\sqrt{1+z^2}}\e^{\nu\log\frac{z}{1+\sqrt{1+z^2}}}}{\sqrt{2\pi
\nu}(1+z^2)^\frac{1}{4}}\left(1+U_1(z)\frac{1}{\nu}+U_2(z)\frac{1}{\nu^2}+U_{3}(z)\frac{1}{\nu^{3}}+O(\frac{1}{\nu^4})\right),
\]
where
\begin{align*}
U_1(z)=&\frac{1}{8\sqrt{1+z^2}}-\frac{5}{24(1+z^2)^\frac{3}{2}},\\
U_2(z)=&\frac{9}{128(1+z^2)}-\frac{77}{192(1+z^2)^2}+\frac{385}{1152(1+z^2)^3},\\
U_3(z)=&\frac{75}{1024(1+z^2)^{\frac{3}{2}}} - \frac{4563}{5120(1+z^2)^{\frac{5}{2}}}+
\frac{17017}{9216(1+z^2)^{\frac{7}{2}}}-\frac{85085}{82944(1+z^2)^{\frac{9}{2}}},
\end{align*}
and 
\[
I_{\nu}'(\nu z)=\frac{(1+z^2)^\frac{1}{4}\e^{\nu\sqrt{1+z^2}}\e^{\nu\log\frac{z}{1+\sqrt{1+z^2}}}}{\sqrt{2\pi
\nu}z}\left(1+V_1(z)\frac{1}{\nu}+V_2(z)\frac{1}{\nu^2}+\dots+O(\frac{1}{\nu^4})\right),
\]
\begin{align*}
V_1(z)=&-\frac{3}{8\sqrt{1+z^2}}+\frac{7}{24(1+z^2)^\frac{3}{2}},\\
V_2(z)=&-\frac{15}{128(1+z^2)}+ \frac{33}{64(1+z^2)^2} - \frac{455}{1152(1+z^2)^3},\\
V_3(z)=&-\frac{105}{1024(1+z^2)^{\frac{3}{2}}} +
\frac{5577}{5120(1+z^2)^{\frac{5}{2}}}-\frac{6545}{3072(1+z^2)^{\frac{7}{2}}}+\frac{95095}{82944(1+z^2)^{\frac{9}{2}}}.
\end{align*}

Using the classical expansion for the logarithm of the Euler Gamma function \cite{GZ} 8.344, we obtain, for large $n$, uniformly in $\lambda$, the expansion of $\log \Gamma(-\lambda,\hat S_n/\mu_{1,n}^2)$ and of $\log \Gamma(-\lambda, S_n/\mu_{1,n}^2)$, and consequentely of the difference

\begin{align*}
\log \Gamma(-\lambda,\hat S_n/\mu_{1,n}^2) &- \log \Gamma(-\lambda, S_n/\mu_{1,n}^2) =
\sum_{h=0}^\infty \left( \hat\phi_{h-1}(\lambda) - \phi_{h-1} (\lambda) \right) \mu_{1,n}^{1-h}\\
&= -\frac{1}{2} \log(1-\lambda) -\frac{1}{2} \log \lambda + \left(\hat \phi_1(\lambda) -
\phi_1(\lambda)\right)\frac{1}{\mu_{1,n}}\\
&+ \left(\hat \phi_2(\lambda) - \phi_2(\lambda)\right)\frac{1}{\mu_{1,n}^2} + \left(\hat
\phi_3(\lambda) - \phi_3(\lambda)\right)\frac{1}{\mu_{1,n}^3} +O\left(\frac{1}{\mu_{1,n}^{4}}\right)
\end{align*}
with
\begin{align*}
\hat\phi_1(\lambda)-\phi_1(\lambda)&=\frac{1}{2}\frac{1}{(1-\lambda)^\frac{1}{2}}-\frac{1}{2}\frac{1}{(1-\lambda)^\frac{3}{2}},\\
\hat\phi_2(\lambda)-\phi_2(\lambda)&=\frac{1}{4}\frac{1}{(1-\lambda)}- \frac{1}{(1-\lambda)^{2}}-\frac{3}{4} \frac{1}{(1-\lambda)^3},\\
\hat\phi_3(\lambda)-\phi_3(\lambda)&=\frac{11}{48}\frac{1}{(1-\lambda)^\frac{3}{2}}-\frac{35}{16}\frac{1}{(1-\lambda)^\frac{5}{2}}
+\frac{67}{16}\frac{1}{(1-\lambda)^\frac{7}{2}}-\frac{107}{48}\frac{1}{(1-\lambda)^\frac{9}{2}}.\\
\end{align*}

Note that the length $\ell$ of the decomposition is precisely $4$. For the $\ec(U_1)=3$, and therefore the larger
integer such that $h-1=\sigma_h\leq 3$ is $4$. However, note that by
Theorem \ref{tt}, only the term with $\sigma_h=1$, and $\sigma_h=3$, namely $h=2,4$, appear in the formula of Theorem \ref{t4}, since the
unique poles of $\zeta(s,U_1)$ are at $s=1$ and $s=3$. We now apply the formulas of Theorem \ref{t4}. 

First, by the definition in equation (\ref{fi1}),
\begin{align*}
\hat \Phi_1(s) - \Phi_1(s)=&\int_0^\infty t^{s-1}\frac{1}{2\pi i}\int_{\Lambda_{\theta,c}}\frac{\e^{-\lambda
t}}{-\lambda}
\left(\frac{1}{2}\frac{1}{(1-\lambda)^\frac{1}{2}}-\frac{1}{2}\frac{1}{(1-\lambda)^\frac{3}{2}}\right) d\lambda dt,\\
\hat \Phi_2(s) - \Phi_2(s)=&\int_0^\infty t^{s-1}\frac{1}{2\pi i}\int_{\Lambda_{\theta,c}}\frac{\e^{-\lambda
t}}{-\lambda}
\left(\frac{1}{4}\frac{1}{(1-\lambda)}- \frac{1}{(1-\lambda)^{2}}-\frac{3}{4} \frac{1}{(1-\lambda)^3}\right)\hspace{-3pt} d\lambda dt,\\
\hat \Phi_3(s) - \Phi_3(s)=&\int_0^\infty t^{s-1}\frac{1}{2\pi i}\int_{\Lambda_{\theta,c}}\frac{\e^{-\lambda
t}}{-\lambda} \left(\frac{11}{48}\frac{1}{(1-\lambda)^\frac{3}{2}}-\frac{35}{16}\frac{1}{(1-\lambda)^\frac{5}{2}}\right) d\lambda dt\\
&+\int_0^\infty t^{s-1}\frac{1}{2\pi i}\int_{\Lambda_{\theta,c}}\frac{\e^{-\lambda
t}}{-\lambda} \left(\frac{67}{16}\frac{1}{(1-\lambda)^\frac{7}{2}}-\frac{107}{48}\frac{1}{(1-\lambda)^\frac{9}{2}}\right) d\lambda dt.\\
\end{align*}

These integrals can be computed using the formula in Appendix \ref{appendixA}. We obtain
\begin{align*}
\Rz_{s=0}\left(\hat \Phi_1(s) - \Phi_1(s)\right)&=-1 ,&\Ru_{s=0}\left(\hat\Phi_1(s) - \Phi_1(s)\right)&=0,\\
\Rz_{s=0}\left(\hat \Phi_2(s) - \Phi_2(s)\right)&=\frac{1}{8} ,&\Ru_{s=0}\left(\hat\Phi_2(s) - \Phi_2(s)\right)&=0,\\
\Rz_{s=0}\left(\hat \Phi_3(s) - \Phi_3(s)\right)&=-\frac{2}{315} ,&\Ru_{s=0}\left(\hat\Phi_3(s) - \Phi_3(s)\right)&=0.
\end{align*}

Second, using this results and the residues of $\zeta(s,U_1)$ given in by equation (\ref{rrr1})
it follows that 
\beq\label{p1}
\begin{aligned}
\hat Z(0)- Z(0)=&-\hat A_{0,1}(0)+ A_{0,1}(0)+\frac{1}{2}\Ru_{s=1}\zeta(s,U_1)\Ru_{s=0}(\hat\Phi_1(s)-\Phi_1(s))\\
&+\frac{1}{2}\Ru_{s=3}\zeta(s,U_1)\Ru_{s=0}(\hat\Phi_3(s)-\Phi_3(s)),\\
=&-\hat A_{0,1}(0)+ A_{0,1}(0),
\end{aligned}
\eeq
and 
\beq\label{p2}
\begin{aligned}
\hat Z'(0)-Z'(0)=&-\hat A_{0,0}(0)-\hat A_{0,1}'(0)+ A_{0,0}(0)+ A_{0,1}'(0)\\
&+\frac{1}{2}\Rz_{s=1}\zeta(s,U_1)\Ru_{s=0}(\hat\Phi_1(s)-\Phi_1(s))\\
&+\frac{1}{2}\Rz_{s=3}\zeta(s,U_1)\Ru_{s=0}(\hat\Phi_3(s)-\Phi_3(s)),\\
=&-\hat A_{0,1}(0)+ A_{0,1}(0)+\frac{1}{2\nu}-\frac{1}{315\nu^3}.
\end{aligned}
\eeq

Third, by equation (\ref{fi2}) and Theorem \ref{tt}, the terms $A_{0,0}(0)$ and $A'_{0,1}(0)$, are
\begin{align*}
A_{0,0}(s)&=\sum_{n=1}^\infty \left(a_{0, 0,n} -b_{1,0,0}u_n^{-1}-b_{3,0,0}u_n^{-3}\right)u_n^{-2 s},\\
A_{0,1}(s)&=\sum_{n=1}^\infty \left(a_{0, 1,n} -b_{1,0,1}u_n^{-1}-b_{3,0,1}u_n^{-3}\right)u_n^{-2 s}.
\end{align*}

Hence, we need the expansion for large $\lambda$ of the functions $\log\Gamma(-\lambda,\hat S_n/\mu_{1,n}^2)$,
$\hat\phi_{1}(\lambda)$, $\hat\phi_{3}(\lambda)$, $\log\Gamma(-\lambda, S_n/\mu_{1,n}^2)$, $\phi_{1}(\lambda)$ and
$\phi_{3}(\lambda)$. Using classical expansions for the Bessel functions and their derivative and the formulas in
equation (\ref{form}), we obtain
\begin{align*}
a_{0,0,n}&=\frac{1}{2}\log 2\pi+\left(\mu_{1,n}+\frac{1}{2}\right)\log\mu_{1,n}-\mu_{1,n}\log 2-\log\Gamma(\mu_{1,n}+1),\\
a_{0,1,n}&=\frac{1}{2}\left(\mu_{1,n}+\frac{1}{2}\right),\\
b_{1,0,0}&=-\frac{1}{12},\hspace{50pt} b_{3,0,0} = \frac{1}{360}, \hspace{50pt}b_{1,0,1}=b_{3,0,1}=0,
\end{align*}
and
\begin{align*}
\hat a_{0,0,n}&=\frac{1}{2}\log 2\pi+\left(\mu_{1,n}+\frac{1}{2}\right)\log\mu_{1,n}-\mu_{1,n}\log 2-\log\Gamma(\mu_{1,n}+1),\\
\hat a_{0,1,n}&=\frac{1}{2}\left(\mu_{1,n}-\frac{1}{2}\right),\\
\hat b_{1,0,0}&=-\frac{1}{12}, \hspace{50pt}\hat b_{3,0,0} = \frac{1}{360}, \hspace{50pt}\hat b_{1,0,1}=\hat b_{3,0,1}=0.
\end{align*}

This  shows that $A_{0,0}(0)=\hat A_{0,0}(0)$, and that
\[
\hat A_{0,1}(s)- A_{0,1}(s)=-\frac{1}{2}\sum_{n=1}^\infty n(n+2)\mu_{1,n}^{-2s}=-\frac{1}{2}\zeta(2s, U_1).
\]

Thus,
\begin{align*}
\hat A_{0,1}(0)- A_{0,1}(0)&=-\frac{1}{4},\\
\hat A'_{0,1}(0)- A'_{0,1}(0)&=\frac{1}{2}\log\nu-\zeta'(-2)-\frac{1}{2}\log 2\pi.
\end{align*}

Substitution in equations (\ref{p1}) and (\ref{p2}), gives
\begin{align*}
\hat Z(0)- Z(0)=&\frac{1}{4},\\
\hat Z'(0)-Z'(0)=&-\frac{1}{2}\log\nu+\zeta'(-2)+\frac{1}{2}\log 2\pi+\frac{1}{2\nu}-\frac{1}{315\nu^3}.
\end{align*}

\subsection{Part II} In this second part we deal with $2Z_0(s)-Z_+(s)-Z_-(s)$. Thus, we consider the sequences $S_0$ and $S_\pm$. The sequence $S_0$ is analogous to the sequence $S$ analyzed in the previous part. We have that 
\begin{align*}
\log \Gamma(-\lambda,S_{0,n}/\mu_{0,n}^2)=&-\log
I_{\mu_{0,n}}(\mu_{0,n}\sqrt{-\lambda})+\mu_{0,n}\log\sqrt{-\lambda}+\mu_{0,n} \log\mu_{0,n}\\ &-\mu_{0,n}\log2
-\log\Gamma(\mu_{0,n}) - \log \mu_{0,n}.
\end{align*}

Using the uniform expansion of $\log I_{\mu_{0,n}}(\mu_{0,n}\sqrt{-\lambda})$, we obtain the uniform expansion for large $n$:
\begin{align*}
\log &\Gamma(-\lambda, S_{0,n}/\mu_{0,n}^2) \\
&= \sum_{h=0}^\infty \phi_{h-1,0} (\lambda) \mu_{0,n}^{1-h}\\
&= \left( - \sqrt{1-\lambda} + \log(1+\sqrt{1-\lambda}) - \log 2 + 1) - \log \sqrt{-\lambda} \right)\mu_{0,n}\\
&+ \frac{1}{4}\log(1-\lambda) +  \left( - U_1(\sqrt{-\lambda})-\frac{1}{12}\right)\frac{1}{\mu_{0,n}}\\
&+ \left( - U_2(\sqrt{-\lambda})+\frac{1}{2}U_1(\sqrt{-\lambda})^2\right) \frac{1}{\mu_{0,n}^2}\\
&+ \left( -
U_3(\sqrt{-\lambda})+U_1(\sqrt{-\lambda})U_2(\sqrt{-\lambda})-\frac{1}{3}U_1(\sqrt{-\lambda})^3+\frac{1}{360}\right)\frac{1}{\mu_{0,n}^3} +
O\left(\frac{1}{\mu_{1,n}^{4}}\right)\hspace{-1.5pt},
\end{align*}
and hence
\begin{align*}
\phi_{1,0}(\lambda)&=-\frac{1}{8} \frac{1}{(1-\lambda)^{\frac{1}{2}}} + \frac{5}{24}\frac{1}{(1-\lambda)^{\frac{3}{2}}}-\frac{1}{12},\\
\phi_{2,0}(\lambda)&=-\frac{1}{16} \frac{1}{(1-\lambda)}+\frac{3}{8} \frac{1}{(1-\lambda)^{2}}-\frac{5}{16} \frac{1}{(1-\lambda)^3},\\
\phi_{3,0}(\lambda)&=-\frac{25}{384} \frac{1}{(1-\lambda)^{\frac{3}{2}}}+\frac{531}{640}
\frac{1}{(1-\lambda)^{\frac{5}{2}}}
-\frac{221}{128}\frac{1}{(1-\lambda)^{\frac{7}{2}}}+\frac{1105}{1152}\frac{1}{(1-\lambda)^{\frac{9}{2}}}+\frac{1}{360}.
\end{align*}


Using the expansion of $\log I_{\mu_{0,n}}(\mu_{0,n}\sqrt{-\lambda})$, and that of the $\phi_{j,0}(\lambda)$ for large $\lambda$, and the definitions in equations (\ref{form}), we compute 

\begin{align*}
a_{0,0,n,0}&=\frac{1}{2}\log 2\pi+\left(\mu_{0,n}+\frac{1}{2}\right)\log\mu_{0,n}-\mu_{0,n}\log 2-\log\Gamma(\mu_{0,n}+1),\\
a_{0,1,n,0}&=\frac{1}{2}\left(\mu_{0,n}+\frac{1}{2}\right),\\
b_{1,0,0,0}&=-\frac{1}{12}, \hspace{50pt}b_{3,0,0,0} = \frac{1}{360}, \hspace{50pt}b_{1,0,1,0}=b_{3,0,1,0}=0.
\end{align*}

The analysis of the sequences $S_\pm$ needs more work. Let define the functions
\[
T^{\pm}_{\nu}(z)=\pm J_{\nu}(z)+zJ'_\nu(z).
\]

Recalling the series definition of the Bessel function
\[
J_\nu(z)=\frac{z^\nu}{2^\nu}\sum_{k=0}^\infty \frac{(-1)^kz^{2k}}{2^{2k}k!\Gamma(\nu+k+1)},
\]
we obtain that near $z=0$
\[
T_\nu^\pm(z) =\left(1\pm\frac{1}{\nu}\right) \frac{z^\nu}{2^\nu\Gamma(\nu)}.
\]

This means that the function $\hat T^\pm_\nu(z)=z^{-\nu} T^\pm_\nu(z)$ is an even function of $z$. Let $z_{\nu,k,\pm}$
be the positive zeros of $T^\pm_\nu(z)$ arranged in increasing order. By the Hadamard factorization theorem, we have
the product expansion
\[
\hat T^\pm_\nu(z)=\hat T^\pm_\nu(z){\prod_{k=-\infty}^{+\infty}}\left(1-\frac{z}{z_{\nu,k,\pm}}\right),
\]
and therefore
\[
T^\pm_\nu(z)=\left(1\pm\frac{1}{\nu}\right)\frac{z^\nu}{2^\nu\Gamma(\nu)}
\prod_{k=1}^{\infty}\left(1-\frac{z^2}{z^2_{\nu,k,\pm}}\right).
\]

Next,  recalling that (when $-\pi<\arg(z)<\frac{\pi}{2}$)
\begin{align*}
J_\nu(iz)&=\e^{\frac{\pi}{2}i\nu} I_\nu(z),\\
J'_\nu(iz)&=\e^{\frac{\pi}{2}i\nu}\e^{-\frac{\pi}{2}i} I'_\nu(z),\\
\end{align*}
we obtain
\[
T_\nu^\pm(iz)=\e^{\frac{\pi}{2}i\nu}\left(\pm I_\nu(z)+zI'_\nu(z)\right).
\]

Thus, we define (for $-\pi<\arg(z)<\frac{\pi}{2}$) \beq\label{pop} Q^\pm_\nu(z)=\e^{-\frac{\pi}{2}i\nu}T_\nu^\pm(i z),
\eeq and hence
\begin{align*}
Q^\pm_\nu(z)&=\pm I_\nu(z)+zI'_\nu(z)=\left(1\pm\frac{1}{\nu}\right)\frac{z^\nu}{2^\nu\Gamma(\nu)}
\prod_{k=1}^{\infty}\left(1+\frac{z^2}{z^2_{\nu,k,\pm}}\right).
\end{align*}

Using these representations, we obtain the following representations for the Gamma functions associated to the
sequences $S_{\pm,n}$. By the definition in equation (\ref{gamma}), with $z=\sqrt{-\lambda}$, we have
\begin{align*}
\log \Gamma(-\lambda,S_{\pm,n})=&-\log\prod_{k=1}^\infty \left(1+\frac{(-\lambda)}{\tilde j_{\mu_{0,n},k,\pm}^2}\right)\\
=&-\log Q^\pm_{\mu_{0,n}}(\sqrt{-\lambda})+\mu_{0,n}\log\sqrt{-\lambda}\\
&-\mu_{0,n}\log 2-\log\Gamma(\mu_{0,n})+\log\left(1\pm\frac{1}{\mu_{0,n}}\right).
\end{align*}

A first consequence of this representations is that we have a complete asymptotic expansion of the Gamma functions
$\log \Gamma(-\lambda,S_{\pm,n})$, and therefore both $S_{+,n}$ and $S_{-,n}$ are sequences of spectral type. Considering the
expansions, it follows that they are both totally regular sequences of infinite order.

Next, we prove that $S_\pm$ are spectrally decomposable over $U$ with power $\kappa=2$ and length $\ell=4$, as in
Definition \ref{spdec}. We have to show that the functions $\log \Gamma(-\lambda,S_{\pm,n}/u_n^2)$, have the
appropriate uniform expansions for large $n$. We have
\begin{align*}
\log \Gamma(-\lambda,S_{\pm,n}/\mu_{0,n}^2)=&-\log
Q^\pm_{\mu_{0,n}}(\mu_{0,n}\sqrt{-\lambda})+\mu_{0,n}\log\sqrt{-\lambda}+\mu_{0,n} \log\mu_{0,n}\\ &-\mu_{0,n}\log2
-\log\Gamma(\mu_{0,n})+\log\left(1\pm\frac{1}{\mu_{0,n}}\right).
\end{align*}

Recalling the expansions given the previous part, we obtain
\begin{align*}
Q^\pm_\nu(\nu z)
&=\sqrt{\nu}(1+z^2)^\frac{1}{4}\frac{\e^{\nu\sqrt{1+z^2}}\e^{\nu\log\frac{z}{1+\sqrt{1+z^2}}}}{\sqrt{2\pi }}\\
&\hspace{30pt}\left(1+W_{1,\pm}(z)\frac{1}{\nu}+W_{2,\pm}(z)\frac{1}{\nu^2}+W_{3,\pm}(z)\frac{1}{\nu^{3}} +
O(\nu^{-4})\right),
\end{align*}
where $p=\frac{1}{\sqrt{1+z^2}}$, and
\begin{align*}
W_{1,\pm}(p)=V_1(p)\pm p,\hspace{15pt}W_{2,\pm}(p)=V_2(p)\pm  pU_1(p),\hspace{15pt}W_{3,\pm}(p)=V_3(p)\pm pU_2(p),
\end{align*}
\begin{align*}
&W_{1,+}(p)=\frac{5}{8}p+\frac{7}{24}p^3,\\
&W_{2,+}(p)=-\frac{1}{128}p^2+\frac{59}{192}p^4-\frac{455}{1152}p^6, \\
&W_{3,+}(p)=-\frac{33}{1024}p^3+\frac{10571}{15360}p^5-\frac{16555}{9216}p^7 + \frac{95095}{82944}p^9,\\
&W_{1,-}(p)=-\frac{11}{8}p+\frac{7}{24}p^3,\\
&W_{2,-}(p)=-\frac{31}{128}p^2+\frac{139}{192}p^4-\frac{455}{1152}p^6,\\
&W_{3,-}(p)=-\frac{177}{1024}p^3+\frac{22891}{15360}p^5-\frac{22715}{9216}p^7 + \frac{95095}{82944}p^9.\\
\end{align*}

This gives,
\begin{align*}
\log\Gamma&(-\lambda, S_{n,\pm}/\mu_{0,n}^2)\\
=&\sum_{h=0}^\infty \phi_{h-1,\pm}(\lambda) \mu_n^{1-h}=\\
=&\left(1-\sqrt{1-\lambda}+\log(1+\sqrt{1-\lambda})-\log 2\right)\mu_{0,n}\\
&-\frac{1}{4}\log(1-\lambda)+\left(-W_{1,\pm}(\sqrt{-\lambda})\pm 1 -\frac{1}{12}\right)\frac{1}{\mu_{0,n}}\\
&+\left(-W_{2,\pm}(\sqrt{-\lambda})+\frac{1}{2}W_{1,\pm}^2(\sqrt{-\lambda})-\frac{1}{2}\right)\frac{1}{\mu^2_{0,n}}\\
&+\left(W_{1,\pm}(\sqrt{-\lambda})W_{2,\pm}(\sqrt{-\lambda})-W_{3,\pm}(\sqrt{-\lambda})-\frac{1}{3}W^{3}_{1,\pm}(\sqrt{-\lambda})\pm\frac{1}{3}+\frac{1}{360}
\right)\frac{1}{\mu_{0,n}^{3}}\\
& + O\left(\frac{1}{\mu_{0,n}^4}\right),
\end{align*}
and hence
\[
\begin{aligned}
\phi_{1,+}(\lambda)&=-\frac{5}{8}\frac{1}{(1-\lambda)^{\frac{1}{2}}}-\frac{7}{24}\frac{1}{(1-\lambda)^{\frac{3}{2}}}+\frac{11}{12},\\
\phi_{1,-}(\lambda)&=\frac{11}{8}\frac{1}{(1-\lambda)^{\frac{1}{2}}}-\frac{7}{24}\frac{1}{(1-\lambda)^{\frac{3}{2}}}+\frac{13}{12},\\
\end{aligned}
\]
\[
\begin{aligned}
\phi_{2,+}(\lambda)&=\frac{3}{16}\frac{1}{1-\lambda}-\frac{1}{8}\frac{1}{(1-\lambda)^2}+\frac{7}{16}\frac{1}{(1-\lambda)^3}-\frac{1}{2},\\
\phi_{2,-}(\lambda)&=\frac{19}{16}\frac{1}{1-\lambda}-\frac{9}{8}\frac{1}{(1-\lambda)^2}+\frac{7}{16}\frac{1}{(1-\lambda)^3}-\frac{1}{2}.\\
\end{aligned}
\]
\[
\begin{aligned}
\phi_{3,+}(\lambda)&=-\frac{17}{384}\frac{1}{(1-\lambda)^{\frac{3}{2}}}-\frac{389}{640}\frac{1}{(1-\lambda)^{\frac{5}{2}}}
+\frac{203}{128}\frac{1}{(1-\lambda)^{\frac{7}{2}}}-\frac{1463}{1152}\frac{1}{(1-\lambda)^{\frac{9}{2}}}+\frac{121}{360},\\
\phi_{3,-}(\lambda)&=\frac{527}{384}\frac{1}{(1-\lambda)^{\frac{3}{2}}}-\frac{1989}{640}\frac{1}{(1-\lambda)^{\frac{5}{2}}}
+\frac{427}{128}\frac{1}{(1-\lambda)^{\frac{7}{2}}}-\frac{1463}{1152}\frac{1}{(1-\lambda)^{\frac{9}{2}}}-\frac{119}{360}.\\
\end{aligned}
\]

By equation (\ref{fi2}) and Theorem \ref{tt}, the terms $A_{0,0}(s)$ and $A_{0,1}(s)$, are
\begin{align*}
A_{0,0,\pm}(s)&=\sum_{n=1}^\infty \left(a_{0, 0,n,\pm} -b_{1,0,0,\pm}u_n^{-1}-b_{3,0,0,\pm}u_n^{-3}\right)u_n^{-2 s},\\
A_{0,1,\pm}(s)&=\sum_{n=1}^\infty \left(a_{0, 1,n,\pm} -b_{1,0,1,\pm}u_n^{-1}-b_{3,0,1,\pm}u_n^{-3}\right)u_n^{-2 s}.
\end{align*}

Hence, we need the expansion for large $\lambda$ of the functions $\log\Gamma(-\lambda,S_{n,\pm}/\mu_{0,n}^2)$,
$\phi_{1,\pm}(\lambda)$ and $\phi_{3,\pm}(\lambda)$. Using equations (\ref{pop}) and the definition, we obtain
\[
Q^\pm_\nu(z)\sim \frac{\sqrt{z}\e^z}{\sqrt{2\pi}}\left(1+\sum_{k=1}^\infty b_kz^{-k}\right)+O(\e^{-z}),
\]
for large $z$. Therefore,
\begin{align*}
\log\Gamma(-\lambda,S_{n,\pm}/\mu_{0,n}^2)=&-\mu_{0,n}
\sqrt{-\lambda}+\frac{1}{2}\left(\mu_{0,n}-\frac{1}{2}\right)\log(-\lambda)
+\frac{1}{2}\log 2\pi\\
&+\left(\mu_{0,n}-\frac{1}{2}\right)\log\mu_{0,n}
-\log 2^{\mu_{0,n}}\Gamma(\mu_{0,n})\\
&+\log\left(1\pm\frac{1}{\mu_n}\right) +O\left(\frac{1}{\sqrt{-\lambda}}\right).
\end{align*}

Thus,
\begin{align*}
a_{0,0,n,\pm}&=\frac{1}{2}\log 2\pi+\left(\mu_{0,n}-\frac{1}{2}\right)\log\mu_{0,n}-\log 2^{\mu_{0,n}}\Gamma(\mu_{0,n})
+\log\left(1\pm\frac{1}{\mu_{0,n}}\right),\\
a_{0,1,n,\pm}&=\frac{1}{2}\left(\mu_{0,n}-\frac{1}{2}\right),\\
b_{1,0,0,+}&=-\frac{11}{12}, \hspace{30pt}b_{3,0,0,+} = \frac{121}{360}, \hspace{30pt}b_{1,0,1,\pm}=b_{3,0,1,\pm} =0,\\
b_{1,0,0,-}&=-\frac{13}{12}, \hspace{30pt}b_{3,0,0,+} = \frac{119}{360}.
\end{align*}

Using these coefficients and the ones obtained for the sequence $S_0$, we conclude that
\begin{align*}
2A_{0,0,0}(s)-A_{0,0,+}(s)-A_{0,0,-}(s) = - \sum_{n=1}^{\infty}
\log\left(1-\frac{1}{\mu^2_{0,n}}\right)\frac{(n+1)^2}{\mu_{0,n}^{2s}},
\end{align*}
and

\begin{align*}
2A_{0,1,0}(s)-A_{0,1,+}(s)-A_{0,1,-}(s) = \sum_{n=1}^{\infty} \frac{(n+1)^2}{\mu_{0,n}^{2s}}.
\end{align*}

Next, we collect the results obtained for giving the uniform expansion of the sum of the logarithmic Gamma functions:
\begin{align*}
2 \log \Gamma(-\lambda,\hat S_{0,n}/\mu_{0,n}^2) - \log\Gamma(-\lambda, S_{n,+}/\mu_{0,n}^2) - \log\Gamma(-\lambda,
&S_{n,-}/\mu_{0,n}^2) \\
&=\sum_{h=1}^\infty \phi_{h-1}(\sqrt{-\lambda})\mu_{0,n}^{1-h},
\end{align*}
where
\[
\phi_{h-1}(\sqrt{-\lambda})=2 \phi_{h-1,0}(\sqrt{-\lambda}) - \phi_{h-1,+}(\sqrt{-\lambda}) -
\phi_{h-1,-}(\sqrt{-\lambda}),
\]
and 
\begin{align*}
\phi_{1}(\sqrt{-\lambda}) &= - \frac{1}{(1-\lambda)^{\frac{1}{2}}} + \frac{1}{(1-\lambda)^{\frac{3}{2}}}, \\
\phi_{2}(\sqrt{-\lambda}) &= - \frac{3}{2} \frac{1}{1-\lambda} + 2 \frac{1}{(1-\lambda)^{2}} - \frac{3}{2} \frac{1}{(1-\lambda)^3} + 1,\\
\phi_{3}(\sqrt{-\lambda}) &=-\frac{35}{24}
\frac{1}{(1-\lambda)^{\frac{3}{2}}} +\frac{43}{8} \frac{1}{(1-\lambda)^{\frac{5}{2}}}
-\frac{67}{8}\frac{1}{(1-\lambda)^{\frac{7}{2}}} +\frac{107}{24} \frac{1}{(1-\lambda)^{\frac{9}{2}}}.
\end{align*}

Let $\Phi_{h-1}(s)=2 \Phi_{h-1,0}(s) - \Phi_{h-1,+}(s) - \Phi_{h-1,-}(s)$. Then, using the definition in equation (\ref{fi1}), and the formula for the integral in Appendix \ref{appendixA}, we have
\begin{align*}
\Phi_{1}(s) &= \frac{2 \Gamma(s+\frac{1}{2})}{\sqrt{\pi}}, \\
\Phi_{2}(s) &= -\frac{\Gamma(s+1)}{2}(5+5s+\frac{3}{2}s^2),\\
\Phi_{3}(s) &= \frac{\Gamma(s+\frac{3}{2})}{\sqrt{\pi}} \left(\frac{428}{315}
+\frac{22}{35}s + \frac{214}{315}s^2\right),
\end{align*}
and hence
\begin{align*}
\Rz_{s=0}\Phi_{1}(s)&=2 ,
&\Ru_{s=0}\Phi_{1}(s)&=0,\\
\Rz_{s=0}\Phi_{2}(s)&=-\frac{5}{2} ,
&\Ru_{s=0}\Phi_{2}(s)&=0,\\
\Rz_{s=0}\Phi_{3}(s)&=\frac{214}{315} , &\Ru_{s=0}\Phi_{3}(s)&=0.
\end{align*}

Using all these results and the residues of the function $\zeta(s,U_0)$ in the formulas given in Theorem \ref{t4}, we obtain
\begin{align*}
2Z_0(0)-Z_+(0)-Z_-(0)=&-2A_{0,1,0}(0)+A_{0,1,+}(0)+A_{0,1,-}(0),\\
2Z'_0(0)-Z'_+(0)-Z'_-(0)=&-2A_{0,0,0}(0)+A_{0,0,+}(0)+A_{0,0,-}(0)
-2A'_{0,1,0}(0)\\
&+A'_{0,1,+}(0)+A'_{0,1,-}(0)
+ \frac{1}{2\nu}\left(1-\frac{1}{\nu^2}\right)+\frac{107}{315\nu^3}.
\end{align*}

Recall that 
\begin{align*}
2A_{0,1,0}(s)-A_{0,1,+}(s)-A_{0,1,-}(s) &= \sum_{n=1}^{\infty} \frac{(n+1)^2}{\mu_{0,n}^{2s}}\\
&=\nu^{-2s}\zeta(2s,U_0)
&=\nu^{-2s}\zeta\left(s,\Sp_+\Delta_{S^3}+\frac{1}{\nu^2}\right),
\end{align*}
and this gives (see \cite{Spr0})
\[
2A_{0,1,0}(0)-A_{0,1,+}(0)-A_{0,1,-}(0) =\zeta\left(0,\Sp_+\Delta_{S^3}+\frac{1}{\nu^2}\right)=-1,
\]
and hence
\[
2Z_0(0)-Z_+(0)-Z_-(0)=-2A_{0,1,0}(0)+A_{0,1,+}(0)+A_{0,1,-}(0)=1.
\]

In order to deal with the other term, it is convenient to proceed as follows. Since, 
\begin{align*}
2A_{0,0,0}(s)-A_{0,0,+}(s)-A_{0,0,-}(s) = -\sum_{n=1}^{\infty}(n+1)^2
\log\frac{\mu_{0,n}^2-1}{\mu_{0,n}^2}\mu_{0,n}^{-2s},
\end{align*}
we have that
\begin{align*}
A(s)=&2A_{0,0,0}(s)-A_{0,0,+}(s)-A_{0,0,-}(s)+2A'_{0,1,0}(s)-A'_{0,1,+}(s)-A'_{0,1,-}(s)\\
=&-\sum_{n=1}^{\infty}(n+1)^2\log(\mu_{0,n}^2-1)\mu_{0,n}^{-2s}.
\end{align*}

Recalling the definition of $\mu_{0,n}$,
\begin{align*}
A(s)=&-\sum_{n=1}^{\infty}(n+1)^2\log(\nu^2n(n+2))\mu_{0,n}^{-2s}\\
=&-2\log \nu\sum_{n=1}^{\infty}(n+1)^2\mu_{0,n}^{-2s}
-\sum_{n=1}^{\infty}(n+1)^2\log(n(n+2))\mu_{0,n}^{-2s}\\
=&-2(\log\nu)\nu^{-2s}\zeta\left(s,\Sp_+\Delta_{S^3}+\frac{1}{\nu^2}\right)+\nu^{-2s}\zeta'\left(s,\Sp_+\Delta_{S^3}+\frac{1}{\nu^2}\right)\\
=&-2(\log\nu)\nu^{-2s}\sum_{j=0}^\infty\binom{-s}{j}\zeta(s+j,\Sp_+\Delta_{S^3})\nu^{-2j}\\
&+\nu^{-2s}\sum_{j=0}^\infty\binom{-s}{j}\zeta'(s+j,\Sp_+\Delta_{S^3})\nu^{-2j},
\end{align*}
and therefore
\begin{align*}
A(0)&=-2\zeta(0,\Sp_+\Delta_{S^3})\log\nu+\zeta'(0,\Sp_+\Delta_{S^3})\\
&=2\log\nu+2\zeta'(-2)+2\zeta'(0)+\log 2.
\end{align*}

This give
\begin{align*}
2Z'_0(0)-Z'_+(0)-Z'_-(0)&=-A(0)\\
&=-2\log\nu-2\zeta'(-2)+\log \pi
+ \frac{1}{2\nu}\left(1-\frac{1}{\nu^2}\right)+\frac{107}{315\nu^3}.
\end{align*}

We can now compute the torsion using equation (\ref{ttt})
\begin{align*}
\log T(C_\alpha S^3_{la})=&\left(\frac{1}{4}+\frac{1}{4}+\frac{1}{2}\right)\log l^2\\
&-\log 2
-\frac{1}{2}\log\nu+\zeta'(-2)+\frac{1}{2}\log 2\pi+\frac{1}{2\nu}-\frac{1}{315\nu^3}\\
&-\log\nu-\zeta'(-2)+\frac{1}{2}\log \pi
+ \frac{1}{4\nu}\left(1-\frac{1}{\nu^2}\right)+\frac{107}{630\nu^3}\\
=&\frac{1}{2}\log \frac{\pi^2l^4}{2\nu^3}+\frac{3}{4}\frac{1}{\nu}-\frac{1}{12\nu^3}.
\end{align*}

We conclude this section reviewing briefly the analysis of the case $n=1$, and $n=2$. All details can be found in \cite{HMS}. In the case $n=1$, the torsion is given by
\[
\log T(C_\alpha S^1_{l\sin\alpha})=\left(\frac{1}{4}+Z(0)-\hat Z(0)\right)\log l^2+Z'(0)-\hat Z'(0)-\frac{1}{2}\log 2,
\]
where 
\[
Z(s)=\sum_{n,k=1}^\infty j_{\nu n,k}^{-2s},\qquad \hat Z(s)=\sum_{n,k=1}^\infty (j'_{\nu n,k})^{-2s}.
\]

Therefore, the analysis is very similar to the one performed in the previous part I, with the main difference that now the zeta function $\zeta(s,U)$ is $\nu^{-s}\zeta(s)$. Therefore, we just have a simple pole at $s=1$, and we only need the expansion of the logarithmic Gamma function up to order $\nu^{-1}$. 

The case of the sphere is a bit more complicate. Now, 
\[
\log T(C_\alpha S^2_{l\sin\alpha})= \left(\frac{3}{4} + \frac{1}{2} X_+(0) - \frac{1}{2} X_-(0) \right) \log l^{2} +
\frac{1}{2}X_+'(0)-\frac{1}{2}X_-'(0)+\frac{1}{2}\log\frac{4}{3},
\]
where
\[
X_+(s)=\sum_{n,k=1}^\infty (2n+1) \hat j_{\mu_n,k}^{-2s},\qquad X_-(s)=\sum_{n,k=1}^\infty (2n+1)\hat j_{\mu_n,k}^{-2s},
\]
$\mu_n=\sqrt{\nu^2 n(n+1)+\frac{1}{4}}$, and  the $\hat j_{\nu,k,\pm}$ are the zeros of the function $G^{\pm}_{\nu}(z)=\pm\frac{1}{2}J_{\nu}(z)+zJ'_\nu(z)$. The zeta function $\zeta(s,U)$ is now related to the zeta function of the Laplace operator on the 2-sphere:
\[
\zeta(2s,U)=\nu^{-2s}\zeta\left(s,\Sp_+\Delta_{S^2}^{(0)}+\frac{1}{4\nu^2}\right).
\]

It is known (see for example \cite{Spr4}), that $\zeta(s,\Sp_+\Delta_{S^2}^{(0)})$ has one simple pole at $s=1$. This gives
\[
\zeta(s,U)=\frac{2}{\nu^2}\frac{1}{s-2}+f(s),
\]
where $f(s)$ is some regular function. Thus,
\begin{align*}
X_+(0) - X_-(0)=&-A_{0,1,+}(0)+A_{0,1,-}(0) + \frac{1}{\nu^2} \Ru_{s=0} (\Phi_{2,+}(s) - \Phi_{2,-}(s))\\
X'_+(0)-X'_-(0)=&-(A_{0,0,+}(0)+A_{0,1,+}'(0)-A_{0,0,-}(0)- A_{0,1,-}'(0))\\
&+\frac{1}{\nu^2}\Rz_{s=0}(\Phi_{2,+}(s)-\Phi_{2,-}(s))\\
&+\left(\frac{\gamma}{\nu^2}+K\right)\Ru_{s=0}(\Phi_{2,+}(s)-\Phi_{2,-}(s)).\\
\end{align*}

Next, proceeding as in the part II above, and introducing the functions
\[
G^\pm_\nu(z)=\pm\frac{1}{2}J_\nu(z)+zJ'_\nu(z),
\]
we obtain the product representation
\begin{align*}
H^\pm_\nu(z)&=\pm\frac{1}{2}I_\nu(z)+zI'_\nu(z)=\left(1\pm\frac{1}{2\nu}\right)\frac{z^\nu}{2^\nu\Gamma(\nu)}
\prod_{k=1}^{\infty}\left(1+\frac{z^2}{z^2_{\nu,k,\pm}}\right),
\end{align*}
where $H^\pm_\nu(z)=\e^{-\frac{\pi}{2}i\nu}G_\nu^\pm(i z)$. This allows to obtain the expansion
\begin{align*}
\log\Gamma(-\lambda, S_{n,\pm}/\mu^2_n)=&\sum_{h=0}^\infty \phi_{h-1,\pm}(\lambda) \mu_n^{1-h}\\
=&\left(1-\sqrt{1-\lambda}+\log(1+\sqrt{1-\lambda})-\log 2\right)\mu_n\\
&-\frac{1}{4}\log(1-\lambda)+\left(-W_{1,\pm}(\sqrt{-\lambda})\pm \frac{1}{2}-\frac{1}{12}\right)\frac{1}{\mu_n}\\
&+\left(-W_{2,\pm}(\sqrt{-\lambda})+\frac{1}{2}W_{1,\pm}^2(\sqrt{-\lambda})-\frac{1}{8}\right)\frac{1}{\mu^2_n}+O\left(\frac{1}{\mu_n^3}\right),
\end{align*}
where $p=\frac{1}{(1-\lambda)^\frac{1}{2}}$, and
\begin{align*}
&W_{1,\pm}(p)=V_1(p)\pm\frac{1}{2}p,&W_{2,\pm}(p)=V_2(p)\pm \frac{1}{2}pU_1(p),
\end{align*}
\begin{align*}
&W_{1,+}(p)=\frac{1}{8}p+\frac{7}{24}p^3,&W_{2,+}(p)=-\frac{7}{128}p^2+\frac{79}{192}p^4-\frac{455}{1152}p^6,\\
&W_{1,-}(p)=-\frac{7}{8}p+\frac{7}{24}p^3,&W_{2,-}(p)=-\frac{28}{128}p^2+\frac{119}{192}p^4-\frac{455}{1152}p^6.\\
\end{align*}

This gives, 
\[
\phi_{2,+}(\lambda)-\phi_{2,-}(\lambda)=-\frac{1}{2}\left(\frac{1}{1-\lambda}-\frac{1}{(1-\lambda)^2}\right),
\]
and hence using the definition in equation (\ref{fi1}),
\begin{align*}
\Phi_{2,+}(s)-\Phi_{2,-}(s)&=-\frac{1}{2}\int_0^\infty t^{s-1}\frac{1}{2\pi i}\int_{\Lambda_{\theta,c}}\frac{\e^{-\lambda t}}{-\lambda}\left(\frac{1}{1-\lambda}-\frac{1}{(1-\lambda)^2}\right).\\
\end{align*}

Using the formula in Appendix \ref{appendixA}, we obtain
\begin{align*}
\Phi_{2,+}(s)-\Phi_{2,-}(s)&=\frac{1}{2}\Gamma(s+1),\\
\end{align*}
and hence
\begin{align*}
&\Rz_{s=0}(\Phi_{2,+}(s)-\Phi_{2,-}(s))=\frac{1}{2},&\Ru_{s=0}(\Phi_{2,+}(s)-\Phi_{2,-}(s))=0.\\
\end{align*}

This gives
\begin{align*}
Z_+(0)-Z_-(0)&= -A_{0,1,+}(0) + A_{0,1,-}(0)\\
Z_+'(0)-Z_-'(0)&=-(A_{0,0,+}(0)+A_{0,1,+}'(0)-A_{0,0,-}(0)- A_{0,1,-}'(0))+\frac{1}{2\nu^2}.\\
\end{align*}

Eventually, using the expansion for large $z$ of the functions $H_\nu^\pm(z)$, we obtain
\begin{align*}
\log\Gamma(-\lambda,S_{n,\pm}/\mu_n^2)=&-\mu_n \sqrt{-\lambda}+\frac{1}{2}\left(\mu_n-\frac{1}{2}\right)\log(-\lambda)
+\frac{1}{2}\log 2\pi\\
&+\left(\mu_n-\frac{1}{2}\right)\log\mu_n
-\log 2^{\mu_n}\Gamma(\mu_n)\\
&+\log\left(1\pm\frac{1}{2\mu_n}\right) +O\left(\frac{1}{\sqrt{-\lambda}}\right),
\end{align*}
and hence
\begin{align*}
a_{0,0,n,\pm}&=\frac{1}{2}\log 2\pi+\left(\mu_n-\frac{1}{2}\right)\log\mu_n-\log 2^{\mu_n}\Gamma(\mu_n)
+\log\left(1\pm\frac{1}{2\mu_n}\right),\\
a_{0,1,n,\pm}&=\frac{1}{2}\left(\mu_n-\frac{1}{2}\right),\\
b_{2,0,0,\pm}&=-\frac{1}{8},\hspace{30pt}b_{2,0,1,\pm}=0.\\
\end{align*}

This immediately shows that $A_{0,1,+}(s)=A_{0,1,-}(s)$, and therefore $X_+(0)-X_-(0) = 0$. Next,
\begin{align*}
A_{0,0,+}(s)-A_{0,0,-}(s)&=\sum_{n=1}^\infty (2n+1) \mu_n^{-2s}\left(\log\left(1+\frac{1}{2\mu_n}\right)-\log\left(1-\frac{1}{2\mu_n}\right)\right)\\
&=F(s,\nu).
\end{align*}

Note that this series converges uniformely for $\Re(s)>2$, but using the analytic extension of the zeta function $\zeta(s,U)$, has an analytic extension that is regular at $s=0$. Therefore,
\begin{align*}
X_+'(0)- X_-'(0)=&-\Rz_{s=0}F(s,\nu)+\frac{1}{2\nu^2}=-\log \frac{\nu^2}{\pi}-f(\nu)+\frac{1}{2\nu^2},
\end{align*}
and this concludes the proof in this case. A power series representation for the function $f(\nu)$ is (see \cite{HMS} Appendix B)
\begin{align*}
f(\nu)=&\log\frac{\nu^2}{\pi}+\zeta(\frac{1}{2},\Sp_+\Delta^{(0)}_{S^2})\frac{1}{\nu}\\
&+\sum_{\substack{j,k=0,\\ j+k\not=0}}^\infty \frac{1}{(2k+1)2^{2k}} \frac{1}{2^{2j}}\binom{-k-\frac{1}{2}}{j}\frac{\zeta(k+j+\frac{1}{2},\Sp_+\Delta^{(0)}_{S^2})}{\nu^{2k+2j+1}}.
\end{align*}

\section{The higher dimensional cases}
\label{s5}

In case of a smooth compact connect Riemannian manifold $(M,g)$ with boundary $\b M$, the analytic torsion is given by the Reidemeister torsion plus some further contributions. It was shown in \cite{Che1}, that this further contribution only depends on the boundary, namely that
\[
\log T(M)=\log\tau(M)+C(\b M).
\]

In the case of a product metric near the boundary, the following formula for this contribution was given by L\"uck \cite{Luc}
\[
\log T(M)=\log\tau(M)+\frac{1}{4}\chi(\b M)\log 2.
\]

In  the general case a further contribution appears, that measures how the metric is {\it far} from a product metric:
\[
\log T(M)=\log\tau(M)+\frac{1}{4}\chi(\b M)\log 2+A(\b M).
\]

A formula for this new {\it anomaly} contribution has been recently given by Br\"uning and Ma \cite{BM}. More precisely, in \cite{BM} (equation (0.6)) is given a formula for the ratio of the analytic torsion of two metrics, $g_0$ and $g_1$,
\beq\label{bat}
\begin{aligned}
\log \frac{T(M,g_1)}{T(M,g_0)}=
\frac{1}{2}\int_{\b M} \left(B(\nabla_1^{T M})-B(\nabla_0^{T M})\right),
\end{aligned}
\eeq
where $\nabla_j^{TM}$ is the curvature form of the metric $g_j$, and the 
forms $B(\nabla_j^{TX})$ are defined in equation
(1.17) of \cite{BM} (see equation \ref{ebm1} below, and observe that we take the opposite sign with respect to the definition in \cite{BM}, since we are considering left actions instead of right actions). Note  that we use the formula of \cite{BM} in the particular case of a flat trivial bundle $F$. Taking $g_1=g$, and $g_0$ an opportune deformation of $g$, that is a product metric near the boundary, 
\[
A(\b M)=\log \frac{T(M,g_1)}{T(M,g_0)},
\]
and therefore

\beq\label{pop1}
\log T(M)=\log\tau(M)+\frac{1}{4}\chi(\b M)\log 2+
\frac{1}{2}\int_{\b M} \left(B(\nabla_1^{T M})-B(\nabla_0^{T M})\right).
\eeq

Since the whole boundary contribution is a local invariant of the boundary, the formula in equation (\ref{bat}) holds in the case of a cone $M=CW$, and therefore in the case under study: $M=C_\alpha S_{l\sin\alpha}^{m-1}$. We compute the contribution given by the formula in equation (\ref{bat}) with respect to the metric induced by the immersion and an opportune product metric. Our result is stated in the following lemma.

\begin{lem} Consider the two metrics 
\begin{align*}
g_1 &= dr \otimes dr + a^2 r^2 g_{S^{n}},\\
g_0& = dr\otimes dr + a^2 l^{2}  g_{S^{n}},
\end{align*}
on $C_\alpha S^{n}_{la}$, where $a=\sin\alpha$. Then, ($p>0$)
\begin{align*}
\log &\frac{T(C_\alpha S^{2p}_{la},g_1)}{T(C_\alpha S^{2p}_{la},g_0)}= \frac{a^{2p}}{8}
\sum_{j=0}^{[p-\frac{1}{2}]} \frac{1}{j!(p-j)!} \sum_{h=0}^{j} \binom{j}{h}
\frac{(-1)^{h}}{(p-j+h)a^{2(j-h)}} \chi(S^{2p}_{la})\\
\log &\frac{T(C_\alpha S^{2p-1}_{la},g_1)}{T(C_\alpha S^{2p-1}_{la},g_0)}\\
&= \sum_{j=0}^{p -1} \frac{2^{p-j}}{j!(2(p-j)-1)!!} \sum_{h=0}^{j} \binom{j}{h}
\frac{(-1)^{h}}{(2(p-j+h)-1)a^{2(j-h)}} \frac{a^{2p-1} (2p-1)!}{4^{p} (p-1)!}
\end{align*}
\label{ele}
\end{lem}

\begin{proof} The proof is a generalization of the proofs of Lemmas 1 and 2 of \cite{HMS}. We first recall some notation from \cite{BZ}  Chapter III and \cite{BM} Section 1.1. For two $\Z/2$-graded algebras $\A$ and $\B$,
let $\A\hat\otimes\B=\A\wedge\hat\B$ denotes the $\Z/2$-graded
tensor product. For two real finite dimensional vector spaces $V$
and $E$, of dimension $m$ and $n$, with $E$  Euclidean and oriented,
the Berezin integral  is the linear map
\begin{align*}
\int^B&: \Lambda V^* \hat\otimes   \Lambda E^* \to \Lambda V^*, \\
\int^B:&\alpha \hat\otimes \beta\mapsto \frac{(-1)^{\frac{n(n+1)}{2}}}{\pi^\frac{n}{2}}\beta(e_1,\dots, e_n)\alpha,
\end{align*}
where $\{e_j\}_{j=1}^n$ is an orthonormal base of $E$. Let $A$ be an antisymmetric endomorphism of $E$. Consider the map
\[
{\hat{}} :A\mapsto \hat A=\frac{1}{2} \sum_{j,l=1}^n (e_j,A e_l)
\hat e^j\wedge \hat e^l.
\]

Note that 
\beq\label{pfpf} 
\int^B \e^{-\frac{\hat A}{2}}=Pf\left(\frac{A}{2\pi}\right), 
\eeq 
and this vanishes if ${\rm dim}E=n$ is odd.

Let $\omega_j$ be the curvature one form over $C_{\alpha}S^{m-1}_{l\sin\alpha}$ associated to the metric $g_j$. Let
$\Theta$ be the curvature two form of the boundary $S^{m-1}$ (with radius 1) and the standard Euclidean metric. Let
$\tensor{(\omega_j)}{^{a}_{b}}$ denotes the entries with line $a$ and column $b$ of the matrix of one forms $\omega_j$.
Then, we introduce the following quantities (see \cite{BM} equations (1.8) and (1.15))
\beq\label{pippo}\begin{aligned}
\mathcal{S}_j&=\frac{1}{2}\sum_{k=1}^{m-1}\tensor{(\omega_j-\omega_0)}{^{r}_{\theta_k}}\hat e^{\theta_k},\\
\hat \Omega&=\mathcal{R}^{T C_{\alpha}S^{m-1}_{l\sin\alpha}}|S^{m-1}_{la}=
\frac{1}{2}\sum_{k,l=1}^{m-1}\tensor{\Omega}{^{\theta_k} _{\theta_l}}
\hat e^{\theta_k}\wedge \hat e^{\theta_l}\\
\mathcal{R}&=\hat \Theta=\frac{1}{2}\sum_{k,l=1}^{m-1}\tensor{\Theta}{^{\theta_k}_{\theta_l}} \hat e^{\theta_k}\wedge \hat e^{\theta_l}.\\
\end{aligned}
\eeq

Direct calculations starting from the metrics $g_j$ allow to obtain explicit formulas for all these forms. The calculations in the present case are a slight generalization of the calculations presented in the proof of Lemma 2 of \cite{HMS}, and we refer to that work for further details. We find that the non zero entries of the matrices appearing in equation (\ref{pippo}) are
\begin{align*}
\tensor{(\omega_1 - \omega_0)}{^{r}_{\theta_i}} &= -  a \prod^{m-1}_{j=i+1} \sin{\theta_j} d\theta_i,\\
\tensor{\Omega}{^{\theta_i}_{\theta_k}}&= (1-a^{2})\prod_{j=i+1}^{k} \sin{\theta_j} \prod_{s=k+1}^{m-1}
\sin^{2}{\theta_s} d\theta_i \wedge d\theta_k, \hspace{20pt}i<k,\\
\tensor{\Theta}{^{\theta_i}_{\theta_k}}&= \prod_{j=i+1}^{k} \sin{\theta_j} \prod_{s=k+1}^{m-1} \sin^{2}{\theta_s}
d\theta_i \wedge d\theta_k, \hspace{20pt}i<k.
\end{align*}

Note that for $i<k$ 
\[
\tensor{((\omega_1 - \omega_0)^2)}{^{\theta_k}_{\theta_i}}  = -a^{2} \prod_{j=i+1}^{k}
\sin{\theta_j} \prod^{m-1}_{s=k+1} \sin^{2}{\theta_s} d\theta_i \wedge d\theta_k. 
\]

Then, recalling $\mathcal{R} = \hat\Omega - 2 \mathcal{S}_1^2$ by equation (1.16) of \cite{BM}, is easy to see
\[
\mathcal{R} =-\frac{2}{a^2} \mathcal{S}_1^2.
\]

Following \cite{BM}, equation (1.17), we define 
\beq\label{ebm1} 
B(\nabla_j^{T C_{\alpha}S^{m-1}_{l\sin\alpha}})=\frac{1}{2}\int_0^1\int^B
\e^{-\frac{1}{2}\mathcal{R}-u^2 \mathcal{S}_j^2}\sum_{k=1}^\infty \frac{1}{\Gamma\left(\frac{k}{2}+1\right)}u^{k-1}
\mathcal{S}_j^k du. 
\eeq

From this definition it follows that $B(\nabla_0^{T C_{\alpha}S^{m-1}_{l\sin\alpha}})$ vanishes identically, since
$\mathcal{S}_0$ does. It remains to evaluate $B(\nabla_1^{T C_{\alpha}S^{m-1}_{l\sin\alpha}})$. Equation
(\ref{ebm1}) gives
\begin{align*}
B(\nabla_1^{T C_{\alpha}S^{m-1}_{l\sin\alpha}})&=\frac{1}{2}\int_0^1\int^B \e^{(\frac{1}{a^2}-u^2)
\mathcal{S}_1^2}\sum_{k=1}^\infty
\frac{1}{\Gamma\left(\frac{k}{2}+1\right)}u^{k-1} \mathcal{S}_1^k du\\
&=\frac{1}{2}\int^B \sum_{j=0,k=1}^\infty
\frac{1}{j!\Gamma\left(\frac{k}{2}+1\right)}\int_0^1 \left(\frac{1}{a^2}-u^2\right)^ju^{k-1} d u \mathcal{S}_1^{k+2j} \\
&=\frac{1}{2}\int^B \sum_{j=0,k=1}^\infty \frac{1}{j!\Gamma\left(\frac{k}{2}+1\right)} \sum_{h=0}^{j}
\binom{j}{h} \frac{(-1)^{h}}{(2h+k)a^{2(j-h)}} \mathcal{S}_1^{k+2j}.
\end{align*}

Since the Berezin integral vanishes identically whenever $k+2j\not=m-1$, we obtain

\beq\label{epe1} 
\begin{aligned}
B&(\nabla_1^{T C_{\alpha}S^{m-1}_{l\sin\alpha}})\\
&=\frac{1}{2} \sum_{j=0}^{[\frac{m}{2} -1]}
\frac{1}{j!\Gamma\left(\frac{m-2j+1}{2}\right)} \sum_{h=0}^{j} \binom{j}{h}
\frac{(-1)^{h}}{(m-2(j-h)-1)a^{2(j-h)}}\int^B \mathcal{S}_1^{m-1}. 
\end{aligned}
\eeq

Now consider the two cases of even and odd $m$ independently. First, assume $m=2p+1$ ($p\geq 0$). Then, using equation
(\ref{pfpf}), equation (\ref{epe1}) gives

\begin{align*}
B(\nabla_1^{T C_{\alpha}S^{2p}_{l\sin\alpha}})
&=\frac{1}{4} \sum_{j=0}^{[p-\frac{1}{2}]} \frac{1}{j!(p-j)!} \sum_{h=0}^{j} \binom{j}{h}
\frac{(-1)^{h}}{(p-j+h)a^{2(j-h)}} \int^B \mathcal{S}_1^{2p}\\
&=\frac{1}{4} \sum_{j=0}^{[p-\frac{1}{2}]} \frac{1}{j!(p-j)!} \sum_{h=0}^{j} \binom{j}{h}
\frac{(-1)^{h}}{(p-j+s)a^{2(j-h)}} \int^B \frac{(-a^2)^p}{2^p}\mathcal{R}^{p}\\
&=\frac{a^{2p}}{4} \sum_{j=0}^{[p-\frac{1}{2}]} \frac{1}{j!(p-j)!} \sum_{h=0}^{j} \binom{j}{h}
\frac{(-1)^{h}}{(p-j+h)a^{2(j-h)}} \int^B \e^{-\frac{\mathcal{R}}{2}}\\
&=\frac{a^{2p}}{4} \sum_{j=0}^{[p-\frac{1}{2}]} \frac{1}{j!(p-j)!} \sum_{h=0}^{j} \binom{j}{h}
\frac{(-1)^{h}}{(p-j+h)a^{2(j-h)}} Pf\left(\frac{\Theta}{2\pi}\right)\\
&=\frac{a^{2p}}{4} \sum_{j=0}^{[p-\frac{1}{2}]} \frac{1}{j!(p-j)!} \sum_{h=0}^{j} \binom{j}{h}
\frac{(-1)^{h}}{(p-j+h)a^{2(j-h)}} e(S^{2p},g_E)\\
\end{align*}
where $e(S^{2p},g_E)$ is the Euler class of $(S^{2p},g_E)$, and we use the fact that
\[
e(S^{2p}_l,g_l)= Pf\left(\frac{\Theta}{2\pi}\right)=\int^{B} \exp(-\frac{\hat\Theta}{2}).
\]

Therefore,

\begin{align*}
\frac{1}{2}\int_{S^{2p}_{l\sin\alpha}}& B(\nabla_1^{T C_{\alpha}S^{2p}_{l\sin\alpha}})\\
 &= \frac{a^{2p}}{8}
 \sum_{j=0}^{[p-\frac{1}{2}]} \frac{1}{j!(p-j)!} \sum_{h=0}^{j} \binom{j}{h}
\frac{(-1)^{h}}{(p-j+h)a^{2(j-h)}} \int_{S^{2p}_{la}} e(S^{2p}_{la},g_E)\\
&= \frac{a^{2p}}{8}  \sum_{j=0}^{[p-\frac{1}{2}]} \frac{1}{j!(p-j)!} \sum_{h=0}^{j} \binom{j}{h}
\frac{(-1)^{h}}{(p-j+h)a^{2(j-h)}} \chi(S^{2p}_{la}).
\end{align*}

Second, assume $m=2p$ ($p\geq 1$). Then, equation (\ref{epe1}) gives

\begin{align*}
B(\nabla_1&^{T C_{\alpha}S^{2p-1}_{l\sin\alpha}})\\
&=\frac{1}{2} \sum_{j=0}^{p -1} \frac{1}{j!\Gamma\left(p-j +
\frac{1}{2}\right)} \sum_{h=0}^{j} \binom{j}{h} \frac{(-1)^{h}}{(2(p-j+h)-1)a^{2(j-h)}} \int^B
\mathcal{S}_1^{2p-1}.
\end{align*}

Now we evaluate $\int^B \mathcal{S}_1^{2p-1}$. Recalling that
\[
\mathcal{R} =-\frac{2}{a^2} \mathcal{S}_1^2,
\] 
we obtain that
\begin{align*}
\int^B \mathcal{S}_1^{2p-1} &= \int^{B} \mathcal{S}_1 \mathcal{S}_1^{2p-2},\\
&= \frac{(-1)^{p-1} a^{2p-2}}{2^{p-1}}\int^{B} \mathcal{S}_1 \mathcal{R}^{p-1}
\end{align*}
and using the esplicit definitions of these forms given in equation (\ref{pippo}), we have

\begin{align*}
\int^B \mathcal{S}_1^{2p-1} =& \frac{(-1)^{p-1} a^{2p-2}}{2^{2p-1}}\int^{B}
\left(\sum_{k=1}^{2p-1}\tensor{(\omega_1-\omega_0)}{^{r}_{\theta_k}} {\hat \e^{\theta_k}}\right)
 \left(\sum_{i,j=1}^{2p-1}\tensor{\Theta}{^{\theta_i}_{\theta_j}}  {\hat \e^{\theta_i}}\wedge  {\hat
 e^{\theta_j}}\right)^{p-1}\\
=&\frac{(-1)^{p} a^{2p-1}}{2^{2p-1}} c_B \\
&\times \sum_{\substack{\sigma \in S_{2p}\\ \sigma(1) = 1}} {\rm sgn}(\sigma)
\tensor{(\omega_1-\omega_0)}{^{1}_{\sigma(2)}} \tensor{(\Omega_0)}{^{\sigma(3)}_{\sigma(4)}} \ldots
\tensor{(\Omega_0)}{^{\sigma(2p-1)}_{\sigma(2p)}}.
\end{align*}
where $c_B=\frac{(-1)^{p(2p-1)}}{\pi^{\frac{2p-1}{2}}}$. Using the same argument used in the final part of the proof of Lemma 2 of \cite{HMS}, we show that
\begin{align*}
\int^B \mathcal{S}_1^{2p-1} &= c_B \frac{(-1)^{p} a^{2p-1} (2p-1)!}{2^{p-1} 2^p} \prod_{j=2}^{2p-1}
(\sin\theta_{j})^{j-1} d\theta_1 \wedge \ldots \wedge d\theta_{2p-1}.
\end{align*}

Then,

\begin{align*}
\int_{S^{2p-1}_{la}} \int^B \mathcal{S}_1^{2p-1} &= \frac{(-1)^{p(2p-1)}}{\pi^{\frac{2p-1}{2}}} \frac{(-1)^{p} a^{2p-1}
(2p-1)!}{2^{p-1} 2^p (la)^{2p-1}} \Vol(S^{2p-1}_{la}) \\
&= \frac{(-1)^{p(2p-1)}}{\pi^{\frac{2p-1}{2}}} \frac{(-1)^{p} a^{2p-1} (2p-1)!}{2^{p-1} 2^p (la)^{2p-1}} \frac{2
\pi^{p} (la)^{2p-1}}{(p-1)!}\\
&=\frac{1}{\pi^{-\frac{1}{2}}} \frac{a^{2p-1} (2p-1)!}{2^{p-1} 2^{p-1} } \frac{1}{(p-1)!}
\end{align*}

and
\begin{align*}
\frac{1}{2}\int_{S^{2p-1}_{l\sin\alpha}} &B(\nabla_1^{T C_{\alpha}S^{2p-1}_{l\sin\alpha}})\\
=& \sum_{j=0}^{p -1} \frac{1}{j!\Gamma\left(p-j + \frac{1}{2}\right)} \sum_{h=0}^{j}\binom{j}{h} \frac{(-1)^{h}}{(2(p-j+h)-1)a^{2(j-h)}} \int_{S^{2p-1}_{l\sin\alpha}} \int^B \frac{\mathcal{S}_1^{2p-1}}{4}\\
=& \sum_{j=0}^{p -1} \frac{1}{j!\Gamma\left(p-j + \frac{1}{2}\right)} \sum_{h=0}^{j}
\binom{j}{h} \frac{(-1)^{h}}{(2(p-j+h)-1)a^{2(j-h)}} \frac{a^{2p-1} (2p-1)!}{\pi^{-\frac{1}{2}} 4^{p}
(p-1)!}\\
=& \sum_{j=0}^{p -1} \frac{2^{p-j}}{j!(2(p-j)-1)!!} \sum_{h=0}^{j} \binom{j}{h}
\frac{(-1)^{h}}{(2(p-j+h)-1)a^{2(j-h)}}\frac{a^{2p-1} (2p-1)!}{4^{p} (p-1)!}.
\end{align*}

\end{proof}

We have now all the terms appearing in equation (\ref{pop1}). In fact, the Reidemeister torsion of the cone over a sphere was computed in \cite{HMS}, Proposition 2, 
\[
\log\tau(C_\alpha S^{m-1}_{l\sin\alpha})=\frac{1}{2}{\rm Vol}(C_\alpha S^{m-1}_{l\sin\alpha}).
\]

Comparing with the results given in Theorem \ref{t1}, we detect the contribution of the singularity. It is easy to see that the formula in equation (\ref{pop1}) holds for the cone over the circle and over the $3$-spheres, while a contribution due to the singularity appears in the case of the sphere. This motivates the following conjecture, that is a theorem for $p<3$. 

\begin{conj} \label{c1} The analytic torsion of the cone $C_\alpha S^{2p-1}_{l\sin\alpha}$, of angle $\alpha$, and length $l>0$, over the odd dimensional sphere $S^{2p-1}$, with the standard metric induced by the immersion in $\R^{m+1}$, and absolute boundary conditions is (where $p>0$):
\begin{align*}
\log T(C_\alpha &S^{2p-1}_{l\sin\alpha})=\frac{1}{2}\log {\rm Vol} (C_\alpha S^{2p-1}_{l\sin\alpha})\\
+& \sum_{j=0}^{p -1} \frac{2^{p-j}}{j!(2(p-j)-1)!!} \sum_{h=0}^{j} \binom{j}{h}
\frac{(-1)^{h}{\rm csc}^{2(j-h)}\alpha}{(2(p-j+h)-1)} \frac{ (2p-1)!\sin^{2p-1}\alpha}{4^{p} (p-1)!}.
\end{align*}

\end{conj}

\section{Appendix A}
\label{appendixA}

We give here a formula for a contour integral appearing in the text. The proof is in \cite{Spr3} Section 4.2. Let
$\Lambda_{\theta,c}=\{\lambda\in\C~|~|\arg(\lambda-c)|=\theta\}$,
$0<\theta<\pi$, $0<c<1$, $a$ real, then
\[
\int_0^\infty t^{s-1} \frac{1}{2\pi
i}\int_{\Lambda_{\theta,c}}\frac{\e^{-\lambda
t}}{-\lambda}\frac{1}{(1-\lambda)^a}d\lambda
dt=\frac{\Gamma(s+a)}{\Gamma(a)s}.
\]

\end{document}